\documentclass[10pt]{amsart}
\usepackage{hyperref}

\usepackage{amsmath,amsfonts,amssymb,amsthm,epsfig,color, tikz, tikz-cd, comment}
\usepackage{graphicx} 
\usepackage{scalerel,stackengine}

\newcommand{\OmMg}{\Omega\mathcal M_g}

\newcommand{\area}{\operatorname{area}}
\newcommand{\Sone}{S^{(1)}}
\newcommand{\Sa}{S^{(a)}}
\newcommand{\omegaone}{\omega^{(1)}}
\newcommand{\omegaa}{\omega^{(a)}}
\newcommand{\Omegaone}{\Omega^{(1)}}
\newcommand{\Omegaa}{\Omega^{(a)}}
\newcommand{\bb}{\mathbb}
\newcommand{\C}{\bb C}

\newcommand{\Z}{\bb Z}
\newcommand{\R}{\bb R}
\newcommand{\N}{\bb N}

\newcommand{\Ell}{\mathcal L}

\newcommand{\covol}{\operatorname{covol}}
\usepackage{float}
\newcommand{\Aff}{\operatorname{Aff}}

\newcommand{\vol}{\operatorname{vol}}

\newcommand{\hh}{\mathcal H}
\newcommand{\cC}{\mathcal C}
\newcommand{\om}{\omega}

\newcommand{\La}{\Lambda}

\newcommand{\cG}{\mathcal G}
\newcommand{\cR}{\mathcal R}
\newcommand{\ir}{\iota_{\cR}}
\newcommand{\Ng}{N_{\cG}}
\newcommand{\Nc}{N_{\cC}}

\newcommand{\zone}{z^{(1)}}
\newcommand{\za}{z^{(a)}}

\newcommand{\JA}{\textcolor{purple}}

\newcommand{\SC}{\operatorname{SC}}

\newcommand{\cone}{c^{(1)}}
\newcommand{\ca}{c^{(a)}}

\newtheorem{Theorem}{Theorem}
\newtheorem{Cor}[Theorem]{Corollary}
\newtheorem{Prop}[Theorem]{Proposition}
\newtheorem{lemma}[Theorem]{Lemma}

\usepackage{amsrefs}

\usepackage[normalem]{ulem}

\numberwithin{equation}{section}
\numberwithin{Theorem}{section}

\begin{document}
\title{Complexity for billiards in regular $N$-gons}
\author[Athreya]{Jayadev S.~Athreya}
\address{Department of Mathematics, University of Washington, Box 354350, Seattle, WA 98195, USA}
\email{jathreya@uw.edu}
\author[Hubert]{Pascal Hubert}
\address{Aix-Marseille Universit\'e,
Institut de Math\'ematiques de Marseille (I2M), UMR 7373, 3 place Victor Hugo, Case 19
13331 Marseille Cedex 3, France}
\email{pascal.hubert@univ-amu.fr}
\author[Troubetzkoy]{Serge Troubetzkoy}
\address{Aix-Marseille Universit\'e,
Institut de Math\'ematiques de Marseille (I2M), UMR 7373, 3 place Victor Hugo, Case 19
13331 Marseille Cedex 3, France}
\email{serge.troubetzkoy@univ-amu.fr}
\begin{abstract} We compute the complexity of the billiard language of the regular Euclidean $N$-gons (and other families of \emph{rational lattice polygons}), answering a question posed in  Cassaigne-Hubert-Troubetzkoy~\cite{CHT}. Our key technical result is a counting result for saddle connections on lattice surfaces, when we count by \emph{combinatorial} length.
\end{abstract}

\maketitle

% Howie comments (13th May 2025)
%1) page 13 in definition of (\epsilon,g) regular shouldn't > \epsilon \ell_g(\gamma) be < in definition?
%2) clarification: measure \sigma_\gamma is not normalized is it?
%3) page 14 with equidistribution. When you build f_0 aren't you assuming WLOG that gamma_0 is horizontal? You say WLOG horizontal only later.
%4) lemma 4.8 shouldn't assumption be \gamma\in A not \gamma \notin A?
%5) 4.13 is a little unclear to me in terms of parentheses and what are multiplicative factors.
%
%6) this is opinion only but in Prop 4.3 I might say there is this constant C_0 such that for all epsilon....
%I was a little confused by the lack of phrase for all \epsilon because you earlier say you fix \epsilon but in a sense you don't. For each epsilon you build a set etc. 

\section{Introduction and main results}\label{sec:intro}

\subsection{Billiards and complexity}\label{subsec:billiards} Let $Q$ be a simply-connected, planar, Euclidean polygon, and let $B_s: Q \times S^1 \rightarrow  Q \times S^1$ denote the billiard flow on $Q$, that is, the movement of a point mass at unit speed on $Q$, with elastic collisions with sides. Labeling the sides of $Q$ by a finite alphabet $\mathcal A$, we define the \emph{language} $\Ell(Q)$ to be the set of all words in $\mathcal A$ which arise as a bounce sequence of a billiard trajectory in $Q$. For $t \in \N$, let $\Ell(t, Q)$ denote the set of words of length $t$ in this language, and let $\rho(Q, t) = \# \Ell(Q, t)$. 
\paragraph*{\bf Regular polygons}
An application of our main results is precise cubic asymptotics for the complexity of the billiard language in the regular $N$-gon $P_N$, a question posed by Cassaigne-Hubert-Troubetzkoy~\cite{CHT}*{\S1}, and also studied experimentally by Davis-Lelievre~\cite{DavisLelievre}*{Conjecture 4.12}:

\begin{Theorem}\label{theorem:billiardcomplexity:Ngons} There are explicit constants $c_N$ such that \begin{equation}\label{eq:billiardcomplexity:Ngons}\lim_{t \rightarrow \infty} \frac{\rho(P_N, t)}{t^3} = c_N.\end{equation} For $N \geq 5$, \begin{equation}\label{eq:cN}c_N = \sigma_N \frac{N^4 \sin^2\left(\frac{2\pi}{N}\right)}{48\pi^2(N-2)} \left( \frac{1}{12} N^2 - \frac{1}{4\sin^2\left(\frac{\pi}{N}\right)} - \frac{1}{12}\right),
\end{equation}
where
  \begin{equation}\label{eq:sigmaN}
  \sigma_N = \begin{cases}
      1 &\mbox{ for }N \mbox{ even}\\
      \frac{4\cos \left( \frac{\pi}{N} \right)}{\left(1+\cos \left( \frac{\pi}{N} \right)\right)^2} &\mbox{ for }N \mbox{ odd}
  \end{cases}
\end{equation}
We have \begin{equation}\label{eq:cN:asymptotics} \lim_{N \rightarrow \infty} \frac{c_N}{N^3} =   \frac{1}{48} \left( \frac 1 3 - \frac{1}{\pi^2} \right).\end{equation}
\end{Theorem}

\paragraph*{\bf Triangles and squares} While there is a constant $c_N$ for all $N$, formula \eqref{eq:cN} does not apply to $N=3, 4$. For $N=3, 4$, $c_3 = 3/(4\pi^2)$ and $c_4 = 4/(\pi^2)$ were computed by Cassaigne-Hubert-Troubetzkoy~\cite{CHT}. Formula \eqref{eq:cN} gives values $3/(8\pi^2)$ and $2/(\pi^2)$ respectively. This factor of two comes from an index 2 difference in the computation of the affine diffeomorphism group in \S\ref{sec:constants:Ngons}, drawn from Veech~\cite{Veech:eisenstein}*{Theorem 1.1}.

\paragraph*{\bf History and prior results} 
The study of complexity of billiards in polygons
began with the square.  Label the sides by two letters, one for vertical sides, the other for horizontal sides,  billiard orbits are
coded by irrational rotations, or Sturmian sequences. These sequences are among the oldest objects studied in symbolic dynamics, going back to at least Morse-Hedlund~\cite{MorseHedlund}. 
Motivated by computational questions, the growth rate of the complexity $\rho(Q,t)$ of the billiard in the square with 
respect to this coding
was computed by Mignosi~\cite{Mignosi} and Berstel-Pochiola~\cite{Berstel}. Rauzy (see, for example~\cite{Rauzy}) developed a program to connect low-complexity sequences, dynamics, and geometry, trying to generalize these links to include substitution systems, interval exchange transformations, and more general billiards. Hubert~\cites{HubertComplexity,HubertTriangles} studied the analogue of the Sturmian coding for arbitrary rational polygons; he showed that billiard orbits
are coded by interval exchange transformations.  
Galperin-Kr\"uger-Troubetzkoy~\cite{GKT} showed that the partition of the phase space is a topological generator up to periodic cylinders, further linking  $\rho(Q,t)$ to the entropy of the billiard. 
Cassaigne-Hubert-Troubetzkoy~\cite{CHT} found a formula linking $\rho(Q,t)$ to the number of \emph{generalized diagonals}, billiard trajectories connecting corners, of
length at most $n$, which implied (using geometric counting results of Masur~\cites{Masur1, Masur2}) that
 for any rational polygon $Q$
there exist strictly positive constants $c,C$ such that 
$$c \le \frac{\rho(Q,t)}{t^3} \le C$$ for all $t \in \N$, and gave a geometric (lattice point counting) based computation of the asymptotic complexity for squares and equilateral triangles, and the authors posed the question if $\rho(Q,t)/t^3$ has a limit (in particular for regular $N$-gons) and if this is the case how to compute it. McMullen~\cite{McMullen:polygons} is an excellent reference for the subtleties of the geometry, dynamics, and number theory of billiards in regular $N$-gons.

\paragraph*{\bf Irrational polygons} In contrast much less is known for irrational polygons. Katok~\cite{Katok} showed that $\rho(Q,t)$ grows at most sub-exponentially for any polygon $Q$. A stretch exponential upper bound has been given by Scheglov for almost every triangle~\cite{Scheglov1} and
a polynomial lower bound for almost every right triangle~\cite{Scheglov2}, which improves the polynomial lower bound given in~\cite{Troubetzkoy} for all polygons.
Recently Krueger-Nogueira-Troubetzkoy~\cite{KrNoTr} showed that 
$\rho(Q,t)$ grows cubically along certain subsequences for generic irrational polygons.

\paragraph*{\bf Rational lattice polygons} Our results apply more generally to \emph{rational lattice polygons} $Q$, polygons whose angles are rational multiples of $\pi$ and whose associated unfolding to a translation surface (together with the data of marked points) is a lattice surface (see \S\ref{subsec:unfolding} for precise definitions). We have:

\begin{Theorem}\label{theorem:billiardcomplexity} Let $Q$ be a rational lattice polygon. There is an (explicit) constant $c_Q$ such that \begin{equation}\label{eq:billiardcomplexity:lattice}\lim_{t \rightarrow \infty} \frac{\rho(Q, t)}{t^3} = c_Q \end{equation}
\end{Theorem}

\paragraph*{\bf Computing constants} For any particular choice of lattice surface $Q$, it is possible to compute the constant $c_Q$ using results of Veech~\cite{Veech:eisenstein} and, for example, the computer program \verb|FlatSurf|. We describe the steps required to compute this at the end of \S\ref{sec:billiards}.

\subsection{Lattice surfaces}\label{subsec:lattices} Our results on complexity for billiards in rational lattice polygons are corollaries of our combinatorial length counting results for lattice translation surfaces. A (compact, genus $g$) translation surface $S$ is a pair $S = (X, \om)$, where $X$ is a compact genus $g$ Riemann surface and $\omega$ a holomorphic $1$-form. We recall that a translation surface is called a \emph{lattice surface} if the set of derivatives $\La^+(S)$ of its (orientation-preserving) affine automorphism group $\operatorname{Aff}^+(S)$ forms a lattice in $SL(2,\R)$. We will give more background and precise definitions in \S\ref{sec:background}.

\paragraph*{\bf Quadratic asymptotics} Let $S :=\left(X,\om\right)$ be a lattice translation surface, and $\cC := \{\gamma_j\}_{j=0}^{m-1}$ a \emph{filling system} of geodesic arcs (that is, $S\backslash \cC$ is a finite union of topological disks, or geometric simply connected polygons). Let $\SC(S)$ denote the set of saddle connections on $S$, and for $\gamma \in \SC(S)$, we define the \emph{combinatorial length} of $\gamma$ with respect to $\cC$ by \begin{equation}\label{eq:comblength} \ell_{\cC}(\gamma) := \sum_{j=0}^{m-1} \iota_{\cG}(\gamma,\gamma_j) \end{equation} where $\iota_{\cG}(\cdot, \cdot)$ is the \emph{geometric intersection number}. For $L >0$, let $$\SC_{\cC}(S, L) := \{\gamma \in \SC(S): \ell_{\cC}(\gamma) \le L\},$$ and $$N_{\cC}(S, L) := \# \SC_{\cC}(S, L).$$ Note that combinatorial length and therefore $N_{\cC}(S, L)$ are independent of the area of $S$. Our main result on $N_{\cC}(S, L)$ is that it has quadratic asymptotics, and that the asymptotic growth can be computed in terms of known asymptotics for \emph{geometric length counting} on the area $1$ scaling $S^{(1)} = (X, \omegaone)$ of $S$. Given a saddle connection $\gamma$, we let $\zone_{\gamma} = \int_{\gamma} \omegaone \in \C$ denote its holonomy vector, and define the \emph{geometric length} \begin{equation}\label{eq:geomlength} \ell_{\cG}(\gamma) := |\zone_{\gamma}|.\end{equation} For $L >0$, let $$\SC_{\cG}\left(\Sone, L\right) := \{\gamma \in \SC(S): \ell_{\cG}(\gamma) \le L\},$$ and $$N_{\cG}\left(\Sone, L\right) := \# \SC_{\cG}\left(\Sone, L\right).$$ 

\begin{equation}\label{eq:combconstantvolume} c_{\cC}(S) := \cone_{\cG}(S) \vol\left(\Omegaone_{S, \cC}\right),\end{equation} where $$\cone_{\cG}(S) := \lim_{L \rightarrow \infty} \frac{N_{\cG}(\Sone, L)}{\pi L^2}$$ is the asymptotic constant for counting by geometric length (whose existence was shown by Veech~\cite{Veech:eisenstein}, see also Vorobets~\cite{Vorobets}*{Theorem 3.13}), and $$\Omegaone_{S, \cC} = \left\{ z \in \C: \sum_{j=0}^{m-1} \left|z \wedge \zone_j\right| \le 1\right\},$$ where $$\zone_j = \int_{\gamma_j} \omegaone$$ is the holonomy vector of $\gamma_j$, and $|z \wedge w|$ is the area of the parallelogram spanned by $z, w \in \C$. This region is the unit ball for a notion of \emph{regularized length}, which is an approximation to combinatorial length, which we will introduce formally below. Our main result is:

\begin{Theorem}\label{theorem:comblengthcounting} Using the above notation we have \begin{equation}\label{eq:comblengthcounting} \lim_{L \rightarrow \infty} \frac{N_{\cC}(S, L)}{L^2} = c_{\cC}(S).\end{equation}
\end{Theorem}

\noindent  We will show in \S\ref{subsec:omega} that $\Omegaone_{S, \cC}$ is always a convex, centrally symmetric polygon with $2(\#\cC) = 2m$ vertices situated in symmetric pairs on the lines $L_j = \{tz_j\}_{t \in \R}$, see Figure~\ref{fig:omegac}.

\begin{figure}
    \centering
\includegraphics[width=0.7\textwidth]{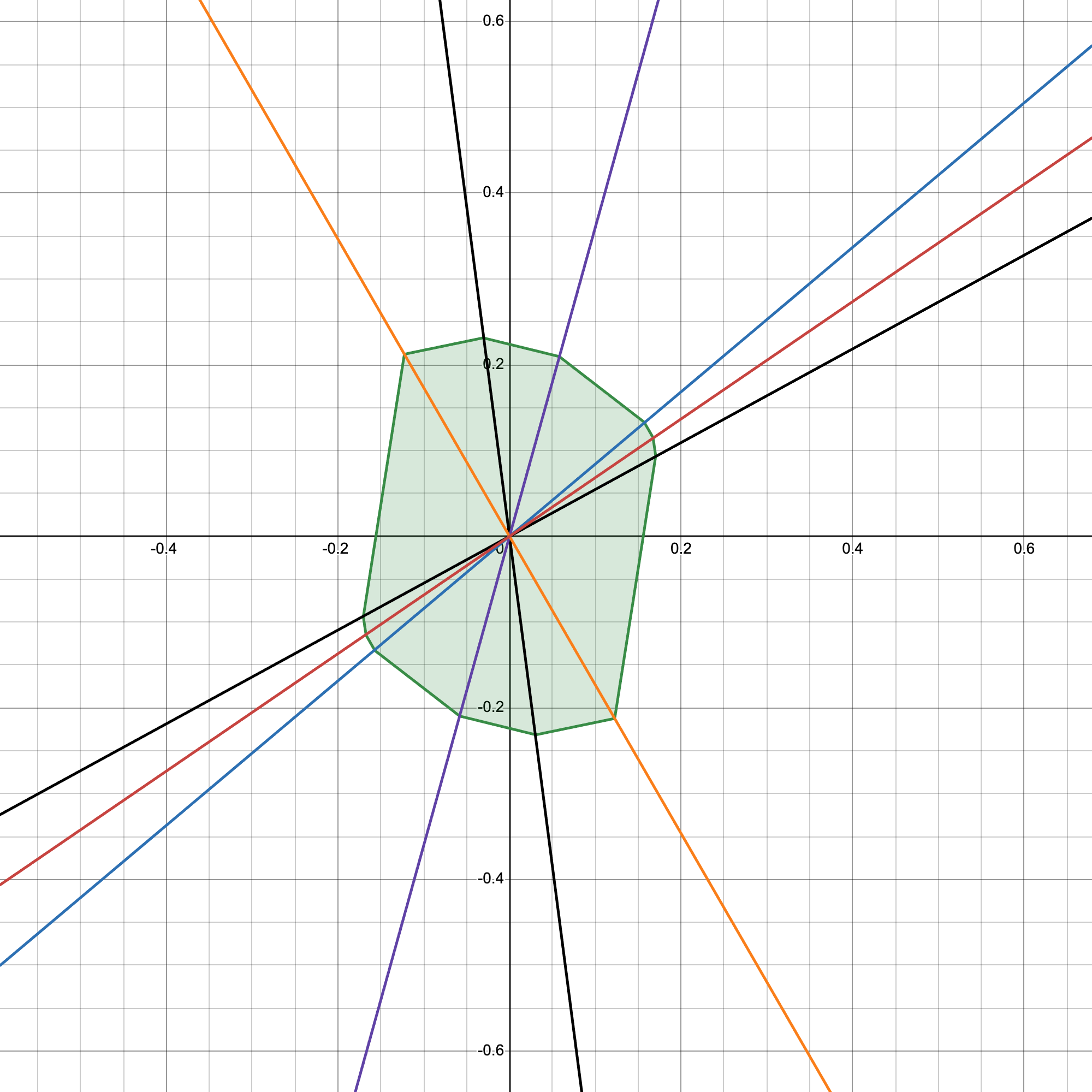}
    \caption{An example region $\Omegaone_{S, \cC}$}
    \label{fig:omegac}
\end{figure}

\paragraph*{\bf Organization} We give background on translation surfaces, lattice surfaces, and billiards in \S\ref{sec:background}, then prove Theorem~\ref{theorem:billiardcomplexity} (assuming Theorem~\ref{theorem:comblengthcounting}) in \S\ref{sec:billiards}. We prove Theorem~\ref{theorem:comblengthcounting} in \S\ref{sec:latticesurfaces}, and show how to prove Theorem~\ref{theorem:billiardcomplexity:Ngons} by computing the constants $c_N$, following a strategy of Veech~\cite{Veech:eisenstein}, in  \S\ref{sec:constants:Ngons}. 

\paragraph*{\bf Acknowledgments} We thank Pierre Arnoux, Artur Avila, Nicolas Bedaride, Julien Cassaigne, Maurice Reichert, Bruce Reznick, and Stefan Steinerberger for useful discussions. We thank Howard Masur for his careful reading of the paper and suggestions to improve the exposition in \S\ref{sec:latticesurfaces}. We thank the anonymous referees for their detailed reports which improved the exposition of the paper. This work was initiated at Centre International de Rencontres Math\'ematiques (CIRM) in Luminy, as part of the Chaire Jean Morlet program in Autumn 2023, and completed in Fall 2024 at CIRM and at Aix-Marseille Universit\'e (AMU). We thank the wonderful staff at CIRM and AMU for their hospitality and the ideal working conditions, and the Chaire Jean Morlet program for the opportunity to start this collaboration. JSA was partially supported by the United States National Science Foundation (DMS 2404705), the Pacific Institute for the Mathematical Sciences, and the Victor Klee Faculty Fellowship at the University of Washington. This work was completed while JSA was visiting the Institut Henri Poincar\'e (IHP) for the trimester program ``Illustration as a Rigorous Mathematical Technique" in Winter 2026. We acknowledge the support of IHP (UAR 839 CNRS-Sorbonne Universit\'e), and LabEx CARMIN (ANR-10-LABX-59-01).

\section{Counting problems for translation surfaces}\label{sec:background} \paragraph*{\bf Translation surfaces} In this section, we recall background on translation surfaces and associated counting problems. Our exposition will draw on ~\cites{AthreyaMasurBook, EskinSurvey}. A compact, genus $g$ \emph{translation surface} $S$ is a pair $S = (X, \omega)$, where $X$ is a compact genus $g$ Riemann surface and $\omega$ a non-zero holomorphic $1$-form. The one-form $\omega$ determines a singular flat metric on $X$, with cone-type singularities at the zeros of $\omega$, with a zero of order $\alpha$ corresponding to cone point of angle $2\pi(\alpha+1)$, and the orders of the zeros must add to $2g-2$. Integrating $\omega$ away from singular points gives an atlas of charts to $\mathbb C$ whose transition maps are translations, leading to the term translation surface.

\paragraph*{\bf Area} 

The \emph{area} of the translation surface $S= (X, \omega)$ is given by $$\area(S) = \frac{i}{2} \int_{X} \omega \wedge \overline{\omega}.$$ For $a >0$, we will write $S^{(a)} = (X, \omega^{(a)}),$ with $$\area\left(S^{(a)}\right) = \frac{i}{2} \int_{X} \omega^{(a)} \wedge \overline{\omega^{(a)}} = a.$$ Note that for $t \in \C$, $$\area(tS) = \area((X, t\omega)) = |t|^2 \area(X, \omega).$$

\paragraph*{\bf Normalizations} As we will discuss further in \S\ref{subsec:areas}, the constant $$c_{\cC}(S) =  \cone_{\cG}(S) \vol\left(\Omegaone_{S, \cC}\right)$$ in Theorem \ref{theorem:comblengthcounting} does not in fact depend on the normalization of the area of the translation surface, although the factors $ \cone_{\cG}(S)$ and $\vol\left(\Omegaone_{S, \cC}\right)$ do. Thus, since we compute $c_{\cC}(S)$ by computing these factors, it will be important for us to specify which normalization of the surface we are using. For the most part, we will use the  standard area one normalization $\Sone$ for the translation surface $S$. On the other hand, in \S\ref{sec:constants:Ngons}, we will be working with regular $N$-gons for which 
the standard normalization is to assume that the vertices are roots of unity $e^{2\pi i j/N}, j = 0, \ldots N-1$. We will be as explicit as possible to avoid any confusion on normalizations.

\paragraph*{\bf Saddle connections, cylinders, holonomy vectors, and area} A \emph{saddle connection} is a geodesic segment with respect to this flat metric which starts and ends at a zero of $\omega$, with no zeros in its interior. A \emph{core curve} of a cylinder is a periodic non-singular geodesic with respect to the flat metric. If $\gamma$ is any arc on the surface (in particular a saddle connection or a cylinder core curve), the associated \emph{holonomy vector} on the area $a$ surface is defined to be \begin{equation}\label{eq:def:holonomy} z^{(a)}_{\gamma} := \int_{\gamma} \omega^{(a)} = \ell_{\cG}(\gamma) e^{i\theta_{\gamma}}.\end{equation} We will write $\ell_{\cG}(\gamma) = |z^{(1)}_{\gamma}| \in \R^+$ for the \emph{geometric length} of $\gamma$ on the area $1$ scaling, and $\theta_{\gamma}$ is the direction of $\gamma$ (which does not depend on scaling).  Note that $z^{(a)}_{\gamma} = \sqrt{a} \zone_{\gamma}$. Let $\SC(S)$ denote the set of saddle connections of a translation surface $S$ (note that this collection does not depend on the scaling of the surface).

\subsection{Strata of translation surfaces}\label{subsec:strata} Let $\OmMg$ denote the moduli space of compact genus $g$ translation surfaces up to biholomorphisms. That is, $S_1 = (X_1, \omega_1) \sim S_2 = (X_2, \omega_2)$ if there is a biholomorphism $f: X_1 \rightarrow X_2$ with $f^* \omega_2 = \omega_1$. We can write $$\OmMg = \bigsqcup_{\alpha} \OmMg(\alpha),$$ where the union is taken over the set of integer partitions $\alpha$ of $2g-2$, and each $\OmMg(\alpha)$ denotes the \emph{stratum} of translation surfaces with zero pattern $\alpha$. Each stratum has at most 3 connected components~\cite{KontsevichZorich}. The group $GL^+(2,\R)$ acts on each $\OmMg(\alpha)$, via post-composition by $\R$-linear maps on $\C$ with the translation atlas of charts to $\C$. The action of the subgroup $SL(2,\R)$ preserves the area (and thus the locus $\mathcal H(\alpha)$ of area $1$ surfaces). We denote the stabilizer (also known as the \emph{Veech group}) of a translation surface by $$SL(S) = \{g \in SL(2,\R): g \cdot S = S\}.$$ This stabilizer group is generically trivial, and always a discrete subgroup of $SL(2,\R)$. This group is the group of \emph{derivatives} of the group $\Aff^+(S)$ of (area-preserving) \emph{affine diffeomorphisms} of the translation surface $S$ (which is also sometimes referred to as the Veech group).

\subsection{Lattice surfaces}\label{subsec:lattice:background} A \emph{lattice surface} (also known as a \emph{Veech surface}) is a translation surface for which $SL(S)$ is a \emph{lattice}, that is a discrete subgroup of finite volume. The lattices $\Gamma$ in $SL(2,\R)$ which arise in this way are always non-uniform, that is, the quotient $SL(2,\R)/\Gamma$ is finite volume but non-compact. References for background on lattice surfaces include~\cite{AthreyaMasurBook}*{\S 7}, Hubert-Schmidt~\cite{HubertSchmidt} and McMullen's recent survey~\cite{McMullen:survey}. An important result of Veech~\cite{Veech:eisenstein} is that for a lattice surface $S$, there is a finite collection $\gamma_1, \ldots, \gamma_j \in \SC(S)$ such that \begin{equation}\label{eq:SCorbit}\SC(S) = \bigcup_{i=1}^j \Aff^+(S) \cdot \gamma_i,\end{equation} that is, the set of saddle connections is a finite union of orbits of the affine diffeomorphism group.

\subsection{Billiards and unfolding}\label{subsec:unfolding} Recall that a polygon $Q \subset \C$ is \emph{rational} if its angles are rational multiples of $\pi$. Equivalently, the group $G(Q)$ generated by reflections in lines through $0$ parallel to the sides of $Q$ is finite.  Via an unfolding process due independently to Fox-Kershner~\cite{FoxKershner} and Zemljakov-Katok~\cite{ZemljakovKatok}, the billiard flow on $Q$ is equivalent to the straight line flow on an associated translation surface $S_Q = (X_Q, \omega_Q)$, tiled by $n(Q) = \#G(Q)$ copies of $Q$. Periodic billiard trajectories on $Q$ result in \emph{cylinders} on $S_Q$, and \emph{generalized diagonals} (billiard trajectories connecting corners of $Q$) on $Q$ lift to \emph{generalized saddle connections} on $S_Q$, flat geodesics connecting pairs of \emph{generalized singular points}: either zeros of $\omega_Q$ or \emph{marked points} arising as non-singular lifts of corners of $Q$.
A rational polygon $Q$ is a \emph{rational lattice polygon} if the group $SL_0(S_Q) := D\Aff^+_{0}(S_Q)$ of derivatives of affine diffeomorphisms of $S_Q$ preserving the set of marked points arising as lifts of corners of $Q$ is a lattice in $SL(2, \R)$. There is a potentially subtle point here: $S_Q$ (without the data of marked points) can be a lattice surface without $Q$ being a rational lattice polygon. For example, for odd \(n\geq 5\), consider the isosceles triangle with angles
\(
\frac{\pi}{n}, \frac{\pi}{n}, \frac{(n-2)\pi}{n}.
\)
The associated translation surface is a lattice surface, while the subgroup of its affine group preserving the regular marked points arising from the two vertices of angle \(\pi/n\) is not a lattice; see Hubert--Schmidt~\cite{HubertSchmidtPolygonal}*{Proposition 3}.

 We write $\SC_0(S_Q)$ for the set of generalized saddle connections on $S_Q$.  
 There is, by \eqref{eq:SCorbit}, a finite collection $\gamma_1, \ldots, \gamma_j \in \SC_0(S_Q)$ \begin{equation}\label{eq:SCorbitQ}\SC_0(S_Q) = \bigcup_{i=1}^j \Aff^+_0(S_Q) \cdot \gamma_i.\end{equation}

\subsection{Counting results}\label{subsec:counting} A natural counting problem is to understand the growth of the set of saddle connections by (geometric) length. Let $$\SC_{\cG}(S, L) = \{\gamma \in \SC(S): \ell_{\cG}(\gamma) \le L\},$$ and $$N_{\cG}(S, L) = \# \SC_{\cG}(S, L).$$ Masur showed that for \emph{every} translation surface $S$ there are quadratic upper and lower bounds for $N_{\cG}(S, L)$, that is, constants $0< c_1(S) \le c_2(S) < \infty$ so that for all $L$, $$c_1(S) L^2 \le N_{\cG}(S, L) \le c_2(S)L^2.$$ 

\paragraph*{\bf Quadratic asymptotics} As discussed in the introduction, Veech~\cite{Veech:eisenstein} showed that if $S$ is a lattice surface, there is a constant $c_{\cG}(S)$ such that $$\lim_{L \rightarrow \infty} \frac{N_{\cG}(S, L)}{\pi L^2} = c_{\cG}(S),$$ and subsequently~\cite{Veech:siegel} showed $L^1$-counting results for the average of $N_{\cG}(S, L)$ as $S$ varies over a stratum $\hh(\alpha)$ of area $1$ translation surfaces. Refining Veech's~\cite{Veech:siegel} strategy, Eskin-Masur~\cite{EskinMasur} showed that for every stratum $\hh= \hh(\alpha)$ there is a constant $c_{\cG}(\hh)$ for \emph{almost every} (with respect to the natural Masur-Smillie-Veech measure $\mu_{\hh}$ on the stratum) $S \in \hh$, $$\lim_{L \rightarrow \infty} \frac{N_{\cG}(S, L)}{\pi L^2} = c_{\cG}(\hh).$$ Eskin-Marklof-Morris~\cite{EskinMarklofMorris} extended Veech's lattice surface results to \emph{branched covers} of lattice surfaces. 

\paragraph*{\bf Billiards} The Eskin-Masur~\cite{EskinMasur} result does not apply to billiards, since the set of surfaces arising from billiards is of measure $0$ in each stratum. Since there are certain families of billiards whose unfoldings yield lattice translation surfaces, and branched covers of lattice translation surfaces, the Veech and Eskin-Marklof-Morris asymptotic counting results do apply to these families of billiards.
\paragraph*{\bf Periodic orbits in regular polygons} Veech~\cite{Veech:polygon} computed the quadratic asymptotics for counting \emph{periodic} trajectories by geometric length for billiards in regular polygons by considering the unfoldings to lattice translation surfaces. For more recent interesting results on the geometry, algebra, and number theory of periodic billiard trajectories in lattice polygons see the works of McMullen~\cites{McMullen:heights, McMullen:polygons}.

\paragraph*{\bf Weak quadratic asymptotics} The most general results on billiards arise from results on \emph{weak quadratic asymptotics} for all translation surfaces. Eskin-Mirzakhani-Mohammadi~\cite{EsMiMo}*{Theorem 2.6 and Theorem 2.12} showed that for every translation surface $S$ (and in particular those arising from billiards), there is a constant $c_{\cG}(S)$ such that $$\lim_{L \rightarrow \infty} \frac{1}{\pi L} \int_0^L N_{\cG}(S, e^t) e^{-2t} dt = c_{\cG}(S).$$ This kind of weak quadratic asymptotics was also studied for billiards in tables whose angles are integer multiples of $\pi/2$ by Athreya-Eskin-Zorich~\cite{AEZ}.

\section{Billiards, unfolding, and complexity}\label{sec:billiards}

\subsection{Complexity and generalized diagonals}\label{subsec:complexity}  We now show how asymptotics of counting saddle connections by combinatorial length can yield information about the asymptotic complexity of billiards, using in a key way a result of Cassaigne-Hubert-Troubetzkoy~\cite{CHT}*{Theorem 1.1} (subsequently generalized by Bedaride~\cite{Bedaride}*{Proposition 2}), which applies to general billiards in polygons. Let $Q$ be any Euclidean polygon, and recall that $\rho(Q, t) = \# \Ell(Q, t)$ is the number of words of length $t$ in the language of the billiard coded by collisions with the sides. Let $N_{\cC}(Q, n)$ denote the number of \emph{(oriented) generalized diagonals} of combinatorial length at most $n$, where a generalized diagonal is a billiard trajectory starting and ending at a corner, and the combinatorial length is given by the number of bounces. 
\begin{Theorem}\label{theorem:cht}\cite{CHT}*{Theorem 1.1}\cite{Bedaride}*{Proposition 2} With notation as above, \begin{equation}\label{eq:CHT} \rho(Q, t) = \sum_{n=0}^{t-1} N_{\cC}(Q, n).
    \end{equation}
\end{Theorem}
\paragraph*{\bf Convexity} This is stated in~\cite{CHT} for \emph{convex} $Q$. Bedaride~\cite{Bedaride} extended it to non-convex $Q$ and non-simply connected polygons,
however his version has a bounded error term. He gave further applications to understanding the complexity of billiards in polyhedra. As noticed in \cite{KrNoTr} the original proof in \cite{CHT} works for non-convex polygons (without an error term).

\paragraph*{\bf From generalized diagonals to complexity}  We now show how Theorem~\ref{theorem:billiardcomplexity} will follow from combining Theorem~\ref{theorem:comblengthcounting} and Theorem~\ref{theorem:cht}. Let $Q$ be a rational lattice polygon, and $S_Q$ the lattice surface associated to its unfolding, tiled by $n_Q$ copies of $Q$, and $\cC_Q$ the collection of arcs on $S_Q$ arising from lifting the sides of $Q$, which by construction decomposes $S_Q$ into a finite union of $n_Q$ copies of $Q$, and is thus filling. 
As discussed in \S\ref{subsec:unfolding}, while some corners of $Q$ may lift to marked points (as opposed to true singular points on $S_Q$), we will refer to geodesic trajectories connecting these generalized singular points as generalized saddle connections on $S_Q$. Indeed, the proof of Theorem~\ref{theorem:comblengthcounting}
applies in this setting, with saddle connections replaced by generalized
saddle connections joining singular or marked points, and with
$\Aff^+(S_Q)$ replaced by $\Aff_0^+(S_Q)$. In particular, no orbits are
discarded: by \eqref{eq:SCorbitQ}, the entire set $\SC_0(S_Q)$ is a finite
union of $\Aff_0^+(S_Q)$-orbits.
Thus, the decomposition \eqref{eq:SCorbitQ} of the set of generalized
saddle connections into $\Aff_0^+(S_Q)$-orbits allows us to apply
Theorem~\ref{theorem:comblengthcounting} to the set $\SC_0(S_Q)$ of
generalized saddle connections arising as lifts of generalized diagonals,
so there is a constant $c^0_{\cC_Q}(S_Q)$ such that
\begin{equation}\label{eq:combQ}
\lim_{L \rightarrow \infty}
\frac{N^0_{\cC_Q}(S_Q,L)}{L^2}
=
c^0_{\cC_Q}(S_Q).
\end{equation}

Since each generalized diagonal on $Q$ lifts to $n_Q$ generalized saddle connections, we have \begin{equation}\label{eq:combQ:billiards}\lim_{n \rightarrow \infty} \frac{N_{\cC}(Q, n)}{n^2} = \frac{1}{n_Q}c^0_{\cC_Q}(S_Q).\end{equation} Putting $c_Q = \frac{1}{3n_Q} c^0_{\cC_Q}(S_Q),$ and combining Theorem~\ref{theorem:cht} with the fact that $$\lim_{t \rightarrow \infty} \frac{1}{t^3} \sum_{n=0}^{t-1} n^2 = \frac{1}{3},$$ we obtain Theorem~\ref{theorem:billiardcomplexity}.

\paragraph*{\bf Computing $c_Q$} To explicitly compute $c_Q$, we can see it suffices to compute $c^0_{\cC_Q}$. Using \eqref{eq:SCorbitQ}, this reduces to finding the saddle connections $\gamma_1, \ldots, \gamma_j$ in the decomposition of the set of saddle connections into finitely many orbits of the affine diffeomorphism group. This can be done by considering a finite (corresponding to the cusps of the affine diffeomorphism group) collection of directions on the lattice surface $S_Q$ in which the surface decomposes into cylinders with saddle connections on their boundaries. For each of these saddle connections, we can compute the asymptotics for its orbit, using a formula of Veech~\cite{Veech:siegel}*{16.8}. Of course, implementing this strategy in any given particular situation can be quite involved, as we will see in our discussion of computing the constants $c_N$ for regular $N$-gons in \S\ref{sec:constants:Ngons} below.

\paragraph*{\bf Weak quadratic asymptotics and averaging} We note that while both weak quadratic asymptotics for \emph{geometric counting} and the relationship between combinatorial counting and complexity are given by a kind of C\'esaro averaging, weak quadratic asymptotics are not sufficient for us to obtain asymptotic combinatorial complexity results.

\section{Lattice surfaces and regularized length}\label{sec:latticesurfaces}
\paragraph*{\bf Strategy} In this section, we fix a lattice surface $S$, and prove Theorem~\ref{theorem:comblengthcounting}. We will show, using the no small (virtual) triangles property of lattice surfaces (introduced by Smillie-Weiss~\cite{SmillieWeiss}) and equidistribution results of Dozier~\cite{Dozier}, that for most saddle connections $\gamma$, we can replace combinatorial length $\ell_{\cC}(\gamma)$ associated to a filling system $\mathcal C = \{\gamma_0, \ldots, \gamma_{m-1}\}$ of saddle connections by the \emph{regularized combinatorial length} (or simply \emph{regularized length}) \begin{equation}\label{eq:regularizedlength} \ell_{\cR}(\gamma) = \sum_{j=0}^{m-1} |\zone_{\gamma} \wedge \zone_j|,\end{equation} 
where it is crucial that we are considering the area $1$ surfaces $\Sone$ (otherwise we would need to divide by $\area(S)$ in \eqref{eq:regularizedlength}). Here $$\zone_{\gamma} = \int_{\gamma} \omegaone = \ell_{\cG}(\gamma)e^{i\theta(\gamma)}, \zone_j = \int_{\gamma_j} \omegaone = \ell_{\cG}(\gamma_j)e^{i\theta(\gamma_j)},$$ so $$|\zone_{\gamma} \wedge \zone_j|= \ell_{\cG}(\gamma)\ell_{\cG}(\gamma_j) \left|\sin(\theta(\gamma)-\theta(\gamma_j))\right|.$$ Let $$\SC_{\ast}(S, L) := \{\gamma \in \SC(S): \ell_{\ast}(\gamma) \le L\},$$ where $\ast$ can be $\cG, \cC, \cR$, and define $$N_{\ast}(S, L) := \# \SC_{\ast}(S, L).$$ We will show:
\begin{equation}\label{eq:combregular}
    \lim_{L \rightarrow \infty} \frac{N_{\cC}(S, L)}{L^2} = \lim_{L \rightarrow \infty} \frac{N_{\cR}(S, L)}{L^2} = c_{\cC}(S).
\end{equation}
The proof of \eqref{eq:combregular} will be completed in
\S\ref{sec:comblengthcountingproof}. The main ingredient, Proposition~\ref{prop:regular},
is proved later in \S\ref{sec:proofregular}.

\subsection{Coarse bounds}\label{sec:coarse} To implement this strategy, we need several lemmas (only one of which uses the lattice property of the translation surface $S$). We first record an elementary lemma giving coarse bounds between our three notions of length (see, for example Troubetzkoy~\cite{Troubetzkoy}).

\begin{lemma}\label{lemma:coarse} Let $\ast = \cC$ or $\cR$. There is a $K = K(S, \cC) \geq 1$ such that for all $\gamma \in \SC(S)$, \begin{equation}\label{eq:coarse} \frac{1}{K} \ell_{\cG}(\gamma) \le \ell_{\ast}(\gamma) \le K \ell_{\cG}(\gamma)
\end{equation}
\end{lemma}

\noindent Since $$ \frac{1}{K} \ell_{\cG}(\gamma) \le \ell_{\ast}(\gamma) \le K \ell_{\cG}(\gamma),$$ we also have  $$ \frac{1}{K} \ell_{\ast}(\gamma) \le \ell_{\cG}(\gamma)  \le K \ell_{\ast}(\gamma).$$ 

\begin{Cor}\label{cor:coarsecounting} Let $\ast = \cC$ or $\cR$. We have  $$\frac{\pi}{K^2} c_{\cG}(S) \le \liminf_{L \rightarrow \infty} \frac{N_{\ast}(S, L)}{L^2}$$ and $$\pi K^2 c_{\cG}(S) \geq \limsup_{L \rightarrow \infty} \frac{N_{\ast}(S, L)}{L^2}.$$ 
\end{Cor}

\begin{proof} By Lemma~\ref{lemma:coarse}, we have that $$\ell_{\cG}(\gamma) \le L/K \Rightarrow \ell_{\ast}(\gamma) \le L,$$ so  $$N_{\cG}(S, L/K) \le N_{\ast}(S, L).$$ Dividing by $L^2$, we obtain $$\frac{N_{\cG}(S, L/K)}{L^2} \le \frac{N_{\ast}(S, L)}{L^2},$$ or $$\frac{1}{K^2}\frac{N_{\cG}(S, L/K)}{(L/K)^2} \le \frac{N_{\ast}(S, L)}{L^2}.$$ Since $$\lim_{L \rightarrow \infty} \frac{N_{\cG}(S, L/K)}{(L/K)^2} = \pi c_{\cG}(S),$$ we have the first inequality. 

To prove the second inequality, we proceed similarly. By Lemma~\ref{lemma:coarse}, we have that $$\ell_{\ast}(\gamma) \le L \Rightarrow \ell_{\cG}(\gamma) \le KL,$$ so  $$N_{\ast}(S, L) \le N_{\cG}(S, KL).$$ Dividing by $K^2L^2$, we obtain $$\frac{N_{\ast}(S, L)}{K^2L^2} \le \frac{N_{\cG}(S, KL)}{K^2L^2},$$ or $$\frac{N_{\ast}(S, L)}{L^2} \le K^2\frac{N_{\cG}(S, KL)}{K^2L^2}.$$ Since $$\lim_{L \rightarrow \infty} \frac{N_{\cG}(S, KL)}{K^2L^2} = \pi c_{\cG}(S),$$ we have the second inequality. 
\end{proof}

\paragraph*{\bf Notation} If $A \subset \SC(S)$ is a subset of the set of saddle connections, we will write $$A_{\ast}(S, L) =\{\gamma \in A: \ell_{\ast}(\gamma) \le L\},$$ and $N_{\ast}(A, S, L) = \#A_{\ast}(S, L),$ for $\ast = \cG, \cR, \cC$. Note that the argument above shows that for $\ast = \cC$ or $\cR$, \begin{equation}\label{eq:coarseA} N_{\ast}(A, S, L) \le N_{\cG}(A, S, KL).\end{equation}

\paragraph*{\bf Regularized intersection} Given two saddle connections $\gamma, \gamma'$ with holonomy vectors $z_{\gamma} = \ell_{\cG}(\gamma)e^{i\theta(\gamma)}, z_{\gamma'} = \ell_{\cG}(\gamma')e^{i\theta(\gamma')}$ we define their \emph{regularized} intersection $\iota_{\cR}(\gamma,\gamma')$ by \begin{equation}\label{eq:regulardef}
\iota_{\cR}(\gamma,\gamma') :=  |z_{\gamma} \wedge z_{\gamma'}| = \ell_{\cG}(\gamma) \ell_{\cG}(\gamma')| \sin(\theta(\gamma)-\theta(\gamma'))|\end{equation} 

\paragraph*{\bf $(\epsilon, \gamma_0)$-regular saddle connections} We now fix a saddle connection $\gamma_0 \in \SC(S)$ and let $\epsilon >0$. We say a saddle connection $\gamma$ is $(\epsilon,\gamma_0)$-\emph{regular} if \begin{equation}\label{eq:gamma0regular}
|\iota_{\cG}(\gamma, \gamma_0) - \iota_{\cR}(\gamma, \gamma_0)| \le \epsilon \ell_{\cG}(\gamma)    
\end{equation}
and define $R^{\epsilon,\gamma_0}$ to be the set of $(\epsilon, \gamma_0)$-regular saddle connections, and let $I^{\epsilon, \gamma_0} := \left( R^{\epsilon,\gamma_0} \right)^c$ be the $(\epsilon,\gamma_0)$-\emph{irregular} saddle connections. The key tool in our proof will be that most saddle connections are $(\epsilon, \gamma_0)$-regular:

\begin{Prop}\label{prop:regular} Fix notation as above. Then there exists $C_0 = C_0(\gamma_0, S)$ (independent of $\epsilon$) such that \begin{equation}\label{eq:regularbound:single}
   \limsup_{L \rightarrow \infty} \frac{N_{\cG}(I^{\epsilon, \gamma_0}, S, L)}{L^2} \le C_0 \epsilon 
\end{equation}
\end{Prop}

\paragraph*{\bf $(\epsilon, \cC)$-regular saddle connections} We will prove Proposition~\ref{prop:regular} in \S\ref{sec:angles} below. We will apply Proposition~\ref{prop:regular} to show that for most saddle connections, regularized and combinatorial lengths are comparable. 
We say a saddle connection $\gamma$ is $(\epsilon,\cC)$-\emph{regular} if \begin{equation}\label{eq:Cregular} \left| \ell_{\cC}(\gamma) - \ell_{\cR}(\gamma)\right| < \epsilon \ell_{\cG}(\gamma)
\end{equation}
and define $R^{\epsilon,\cC}$ to be the set of $(\epsilon, \cC)$-regular saddle connections, and let $I^{\epsilon, \cC} := \left( R^{\epsilon,\cC} \right)^c$ to be the $(\epsilon,\cC)$-\emph{irregular} saddle connections. Since $\# \cC = m$, $$\cap_{j=0}^{m-1} R^{\epsilon/m, \gamma_j} \subseteq R^{\epsilon, \cC},$$ and so $$I^{\epsilon, \cC} \subseteq \bigcup_{j=0}^{m-1} I^{\epsilon/m, \gamma_j}.$$ Thus we have, as a corollary to Proposition~\ref{prop:regular}:
\begin{Cor}\label{cor:cregular} Fix notation as above. There is a $C_{\cC}>0$ such that \begin{equation}\label{eq:regularbound}
   \limsup_{L \rightarrow \infty} \frac{N_{\cG}(I^{\epsilon, \cC}, S, L)}{L^2} \le C_{\cC} \epsilon 
\end{equation} 
\end{Cor}

\paragraph*{\bf Regularized counting} Our final tool in the proof of Theorem~\ref{theorem:comblengthcounting} is a consequence of a counting theorem of Veech~\cite{Veech:siegel}*{Theorem 0.18}. 

\begin{Theorem}\label{theorem:regularcounting} Fix notation as above. Then \begin{equation}\label{eq:regularcounting} \lim_{L \rightarrow \infty} \frac{N_{\cR}(S, L)}{L^2} = c_{\cC}(S).\end{equation}
\end{Theorem}

\begin{proof} Let $F := \chi_{\Omega_{\cC}^{(1)}}$ be the indicator function of the region $\Omega_{\cC}^{(1)}.$ Fix $\epsilon >0$, and let $\Psi_{\pm\epsilon}: \C \rightarrow [0, 1]$ be non-negative, continuous compactly supported functions on $\C$ such that $$\Psi_{-\epsilon} < F < \Psi_{\epsilon},$$ and $$\||F-\Psi_{\pm \epsilon}\|_{L^1} < \epsilon,$$ such that the support of $\Psi_{\epsilon}$ is contained in $(1+\epsilon) \Omegaone_{S, \cC}$, $\Psi_{\epsilon} \equiv 1$ on $\Omegaone_{S, \cC}$, and the support of $\Psi_{-\epsilon}$ is contained in $\Omegaone_{S, \cC}$, and  $\Psi_{-\epsilon} \equiv 1$ on $(1-\epsilon)\Omegaone_{S, \cC}$. By definition, $$N_{\cR}(S, L) = \sum_{\gamma \in \SC(S)} F\left(\frac{z_{\gamma}}{L}\right).$$ Applying Veech~\cite{Veech:siegel}*{Theorem 0.18} to $\Psi_{\pm \epsilon}$ yields $$\lim_{L \rightarrow \infty} \frac{1}{L^2}\sum_{\gamma \in \SC(S)} \Psi_{\pm \epsilon}\left(\frac{z_{\gamma}}{L}\right) = \cone_{\cG} \int_{\C} \Psi_{\pm \epsilon}(z) dm(z),$$ where $dm$ is Lebesgue measure on $\C$. Thus, we have $$(1-\epsilon) c_{\cC}(S) \le \lim_{L \rightarrow \infty} \frac{N_{\cR}(S, L)}{L^2} \le (1+\epsilon) c_{\cC}(S),$$ and since $\epsilon$ was arbitrary, we have our result. 
\end{proof}

\subsection{Proof of Theorem~\ref{theorem:comblengthcounting}}\label{sec:comblengthcountingproof} We will show how to prove Theorem~\ref{theorem:comblengthcounting} using Corollary~\ref{cor:cregular} and Theorem~\ref{theorem:regularcounting}. By definition, $$\SC_{\cC}(S, L) = R^{\epsilon, \cC}_{\cC}(S, L) \bigsqcup I^{\epsilon, \cC}_{\cC}(S, L).$$ For $\gamma \in R^{\epsilon, \cC}_{\cC}(S, L)$, 
\begin{align*} \ell_{\cR}(\gamma) &\le \ell_{\cC}(\gamma) + \epsilon \ell_{\cG}(\gamma) \\
    & \le \ell_{\cC}(\gamma) + \epsilon K \ell_{\cC}(\gamma) \\ &= (1+\epsilon K)L,
\end{align*}
where the last inequality follows from Lemma~\ref{lemma:coarse}. So $$N_{\cC}(R^{\epsilon, \cC}, S, L) \le N_{\cR}(S, (1+\epsilon K)L),$$ and thus \begin{equation}\label{eq:combupperbound}
\limsup_{L \rightarrow \infty} \frac{N_{\cC}(S, L)}{L^2} \le \limsup_{L \rightarrow \infty} \frac{N_{\cR}(S, (1+\epsilon K)L)}{L^2} + \limsup_{L \rightarrow \infty} \frac{N_{\cC}(I^{\epsilon, \cC}, S, L)}{L^2}.
\end{equation}
Using Theorem~\ref{theorem:regularcounting} and Corollary~\ref{cor:cregular}, the right-hand side of \eqref{eq:combupperbound} is bounded above by 
$$c_{\cC}(S)(1+\epsilon K)^2 + K^2 C_{\cC}\epsilon,$$
 and since $\epsilon >0$ was arbitrary, we have $$\limsup_{L \rightarrow \infty} \frac{N_{\cC}(S, L)}{L^2} \le c_{\cC}(S).$$ For the other direction note that for $\gamma \in R^{\epsilon, \cC}$ \begin{align*}
\ell_{\cR}(\gamma) &\geq \ell_{\cC}(\gamma) - \epsilon \ell_{\cG}(\gamma) \\ &\geq \left(1-\epsilon K\right) \ell_{\cC}(\gamma),  
\end{align*} since $\ell_{\cG}(\gamma) \le K \ell_{\cC}(\gamma)$. Thus $$R_{\cR}^{\epsilon, \cC}\left(S, L\left(1-\epsilon K\right)\right) \subset \SC_{\cC}(S, L).$$ By definition $$ N_{\cR}\left(S,  L\left(1-\epsilon K\right)\right) - N_{\cR}\left(R^{\epsilon, \cC}, S,  L\left(1-\epsilon K\right)\right)   = N_{\cR}\left(I^{\epsilon, \cC}, S,  L\left(1-\epsilon K\right)\right),$$ and by \eqref{eq:coarseA}, we have $$N_{\cR}\left(I^{\epsilon, \cC}, S,  L\left(1-\epsilon K\right)\right) \le N_{\cG}\left(I^{\epsilon, \cC}, S,  LK\left(1-\epsilon K\right)\right).$$ Therefore, \begin{align*}
\liminf_{L \rightarrow \infty} \frac{N_{\cC}(S, L)}{L^2} &\geq \liminf_{L \rightarrow \infty} \frac{N_{\cR}\left(S, L\left(1-\epsilon K \right)\right)}{L^2} - \limsup_{L \rightarrow \infty} \frac{N_{\cG}\left(I^{\epsilon, \cC}, S, LK\left(1-\epsilon K\right)\right)}{L^2} \\ 
&\geq c_{\cC}(S)\left(1-\epsilon K\right)^2 - C_{\cC}(S) K^2 \left(1-\epsilon K\right)^2\epsilon,
\end{align*}
where the last inequality follows from \eqref{eq:regularbound} and Theorem~\ref{theorem:regularcounting}. Again, since $\epsilon$ was arbitrary, we have $$\liminf_{L \rightarrow \infty} \frac{N_{\cC}(S, L)}{L^2} \geq c_{\cC}(S),$$ which concludes the proof. \qed\medskip

\subsection{Angles and intersections}\label{sec:angles} We now turn to proving Proposition~\ref{prop:regular}, using a series of lemmas. We only use the lattice property in the last lemma, whereas the first two rely on counting and equidistribution results of Dozier~\cite{Dozier}. We recall that we fix a saddle connection $\gamma_0 \in \SC(S)$, and $\epsilon>0$. Let $$A^{\epsilon, \gamma_0} := \{\gamma \in \SC(S): |\sin(\theta(\gamma) - \theta(\gamma_0))| < \epsilon\}$$ be the set of saddle connections whose angle is within $\epsilon$ of $\gamma_0$. Our next lemma says that this set is not large.
\begin{lemma}\label{lemma:sectors} There is a $c_A = c_A(S)$ (independent of $\gamma_0$) such that for all $L >0$, \begin{equation}\label{eq:sectors} N_{\cG}(A^{\epsilon, \gamma_0}, S, L) \le c_A \epsilon L^2.
\end{equation}   
\end{lemma}
\begin{proof} This is an immediate corollary of Dozier~\cite{Dozier}*{Theorem 1.8}.
\end{proof}

\paragraph*{\bf (Non)-equidistribution} Our next lemma concerns the size of the set of saddle connections which have non-equidistributing behavior. Let $\lambda$ denote the Lebesgue probability measure on $S$, and let $g: S \rightarrow \R$ be a continuous function with mean $0$ with respect to $\lambda$. Given a saddle connection $\gamma$, let $\sigma_{\gamma}$ denote the measure on $S$ given by length Lebesgue measure supported on $\gamma$, that is, $\sigma_{\gamma}(S) = \sigma_{\gamma}(\gamma) = \ell_{\cG}(\gamma)$. For $\epsilon >0$, we say a saddle connection $\gamma$ is $(\epsilon, g)$-\emph{regular} if $$\left|\int_S g d\sigma_{\gamma}\right| < \epsilon \ell_{\cG}(\gamma),$$ and denote the set of $(\epsilon, g)$-regular saddle connections by $R^{\epsilon, g}$, and its complement by $I^{\epsilon, g}$. 
\begin{lemma}\label{lemma:nonequidist} Fix notation as above. Then \begin{equation}\label{eq:nonequidist} \limsup_{L \rightarrow \infty} \frac{N_{\cG}(I^{\epsilon, g}, S, L)}{L^2} = 0.
\end{equation}
\end{lemma}
\begin{proof} We define the set $I^{\epsilon ,g, +}$ of \emph{positively irregular} saddle connections $\gamma$, that is, those with $$\int_S g d\sigma_{\gamma} > \epsilon \ell_{\cG}(\gamma),$$ and show $$\limsup_{L \rightarrow \infty} \frac{N_{\cG}(I^{\epsilon, g, +}, S, L)}{L^2} = 0.$$ The same argument will also apply to the similarly defined set of \emph{negatively irregular} saddle connections. We will use Dozier's proof of \cite{Dozier}*{Theorem 1.3} and adapt it for subsequences. Suppose by contradiction that there is a sequence $L_k \rightarrow \infty$, and a $c>0$ so that $$N_{\cG}(I^{\epsilon, g, +}, S, L_k) \geq cL_k^2.$$ Define the averaged measure $$\beta_k = \frac{1}{N_{\cG}(I^{\epsilon, g, +}, S, L_k)} \sum_{\gamma \in I^{\epsilon, g, +}(S, L_k)} \frac{\sigma_{\gamma}}{\ell_{\cG}(\gamma)}.$$  Note that Dozier uses the notation $\mu_{R, S}$ for the measure obtained by averaging the probability measures on saddle connections of length at most $R$ in a subset S. Since we use $S$ to denote the surface, and $\mu$ in other ways, we have changed the notation. Following Dozier~\cite{Dozier}*{page 106, proof sketch of Theorem 1.3}, there is a limit measure on $T^1 S = S \times S^1$ whose projection to the $S^1$-factor is absolutely continuous with respect to Lebesgue measure on $S^1$, which implies that the projection of the limit measure to $S$ must be invariant under the directional flow in a positive Lebesgue measure set of angles, which, by Kerckhoff-Masur-Smillie~\cite{KMS} implies that any limit measure on $S$ must be Lebesgue, but this is contradicted by the definition of irregularity.
\end{proof}

\begin{figure}[ht!]
    \centering
    \begin{tikzpicture}[scale=2.0]
        \draw(0,0)--(4.1,0)node[right, below]{\tiny $R$}--(4.1,0.3)--(0, 0.3)--cycle;
        \draw[<-](4.2, 0)--(4.2, 0.07)node[above]{\tiny $b_1$};
        \draw[->](4.2, 0.23)--(4.2, 0.3);
        
  \draw[<-](0, 0.34)--(.07, 0.34)node[right]{\tiny $b_2$};
    \draw[->](.23, 0.34)--(.3, 0.34);
        \draw(0.3,0.15)--(3.8,0.15)node[right]{\tiny $\gamma_0$};
        \filldraw [black] (0.3,0.15) circle (.5pt);
        \filldraw [black] (3.8,0.15) circle (.5pt);
        \draw[red](1,0)--(1.25, 0.3);
        \draw[red](0.5,0)--(.75, 0.3);
        \draw[red](1.1,0)--(1.35, 0.3);
        \draw[red](1.6,0)--(1.85, 0.3);
        \draw[red](3,0)--(3.25, 0.3);
        \draw[red](3.65,0)--(3.9, 0.3);
        \draw[dashed](0.3,0)--(0.3, 0.3);
           \draw[dashed](3.8,0)--(3.8, 0.3);
             \draw[<-](3.8, 0.34)--(3.87, 0.34)node[right]{\tiny $b_2$};
    \draw[->](4.03, 0.34)--(4.1, 0.34);

    \end{tikzpicture}
    \caption{The intersections of the \textcolor{red}{saddle connection $\gamma$} with our fixed saddle connection $\gamma_0$, compared to the crossings across the region $R =R(\gamma_0, b_1, b_2)$, as measured by integrating a bump function $f_0$ whose support is contained in $R$. If the angle $\theta(\gamma) = \theta$, the length of each red crossing segment is $\frac{b_1}{|\sin(\theta)|}$.}
    \label{fig:spacings}
\end{figure}

\paragraph*{\bf Spacings and the lattice property} We now turn to our final lemma, Lemma~\ref{lemma:intersections}, which states that for saddle connections which make a definite angle with a fixed saddle connection, we can approximate the geometric intersection number using a continuous version of the indicator function of a saddle connection. We use the lattice property crucially in the proof of the intermediate result Lemma~\ref{lemma:spacings} which we need for the proof of Lemma~\ref{lemma:intersections}. We fix a saddle connection $\gamma_0$ on $S$, and fix $\epsilon>0$. We will assume (without loss of generality) that $\gamma_0$ is a horizontal saddle connection, that is $\theta(\gamma_0) =0$. 

\paragraph*{\bf A continuous, mean-zero approximation} Let $f_0:S\to\R$ be a continuous function supported in a sufficiently
small neighborhood $R=R(\gamma_0,b_1,b_2)$ of $\gamma_0$.  We identify
the part of $R$ away from the endpoints of $\gamma_0$ with
\[
[0,\ell_{\cG}(\gamma_0)+2b_2]\times[0,b_1],
\]
so that
\[
\gamma_0=[b_2,\ell_{\cG}(\gamma_0)+b_2]\times\{b_1/2\}
\]
see Figure~\ref{fig:spacings}.  Near a singular endpoint of $\gamma_0$,
Figure~\ref{fig:spacings} is understood sector by sector in the
corresponding Euclidean charts.
Fix $0<\eta<b_2$. Let $E_-$ and $E_+$ denote, respectively, the
portions of $R$ with horizontal coordinate
\[
x\leq b_2+\eta
\qquad\text{and}\qquad
x\geq \ell_{\cG}(\gamma_0)+b_2-\eta;
\]
near a singular endpoint these are understood sector by sector in the
corresponding Euclidean charts.
Thus, away from $E_-\cup E_+$, the region $R$ is the Euclidean
rectangle
\[
[b_2+\eta,\ell_{\cG}(\gamma_0)+b_2-\eta]\times[0,b_1].
\]
We choose $f_0$ to be equal to $1/b_1$ on this rectangle, and choose it
continuously on $E_-\cup E_+$, with $0\leq f_0\leq 1/b_1$, vanishing
outside $R$, so that
\[
\int_S f_0\,d\lambda=\ell_{\cG}(\gamma_0).
\]
We choose $b_1,b_2$ sufficiently small so that
$\max(b_1/b_2,2b_2)<\epsilon$ and $R$ contains no zeros of $S$ other
than the endpoints of $\gamma_0$.  Since each endpoint has only
finitely many Euclidean sectors, the number of components of
$E_-\cup E_+$ is bounded in terms of $S$ alone.  We think of $f_0$ as
a continuous approximation to the measure $\sigma_{\gamma_0}$.
In particular,
\[
g_0:=f_0-\ell_{\cG}(\gamma_0)
\]
is a continuous function of mean zero with respect to $\lambda$.

\begin{lemma}\label{lemma:intersections} There is a $\tilde{C} = \tilde{C}(S)$ so that for any $\gamma \notin A^{\epsilon, \gamma_0}$ (that is, $|\sin(\theta(\gamma))|= |\sin(\theta(\gamma) - \theta(\gamma_0))|> \epsilon$), we have \begin{equation}\label{eq:intersections} \left| \left|\sin\left(\theta(\gamma)\right)\right|\cdot \int_{S} f_0 d\sigma_{\gamma} - \iota_{\cG}(\gamma, \gamma_0) \right| \le \tilde{C} \epsilon \ell_{\cG}(\gamma).
\end{equation}
\end{lemma}

\paragraph*{\bf Spacings on horizontal segments} To prove Lemma~\ref{lemma:intersections} we first require a lemma on spacings of intersection points of saddle connections and horizontal segments on lattice surfaces. 

\begin{lemma}\label{lemma:spacings} Let $I$ be a horizontal segment of length $L$ on $S$, where we identify $I$ with the interval $[0, L]$. Let $\gamma$ be a saddle connection on $S$, and let $\Gamma := \{\gamma_1, \ldots, \gamma_N\}$ be the set of saddle connections in direction $\theta = \theta(\gamma)$, and let $$0 \le t_1 < t_2 < \ldots < t_k \le L$$ be the points of intersection of the collection $\{\gamma_i\}_{i=1}^N$ with $I$, $k = \sum_{j=1}^N \iota_{\cG}(\gamma_j, I)$. Then there is a $C = C(S)>0$ such that \begin{equation}\label{eq:spacings} \min_j (t_{j+1}- t_j) \geq \frac{C}{|\sin(\theta)| \ell_{\cG}(\gamma)}\end{equation}

\end{lemma}
\begin{proof} Since $S$ is a lattice surface, in the direction $\theta = \theta(\gamma)$ of $\gamma$, it is decomposed into a finite number of cylinders, $\{\cC_i\}_{i=1}^J$, where $\cC_i$ is an isometrically embedded copy of $(\R/w_i\Z) \times (0, h_i)$, where $w_i$ is the geometric length of the core curve $\eta_i$ of $\cC_i$, which we refer to as the \emph{width} and we refer to $h_i$ as the \emph{height} of the cylinder. The intersection of the interiors of the cylinders $\cC_i$ in direction $\theta(\gamma)$ with $I$ is a finite number of open intervals. The length of a horizontal interval crossing $\cC_i$ is
$$h_i' := h_i/|\sin(\theta)|,$$ see Figure~\ref{fig:crossing}. The spacings $t_{j+1}-t_j$ are precisely given by crossings of cylinders in direction $\theta$, that is, (some of) these $h_i' = h_i/|\sin(\theta)|$. Let $a_i := a(\cC_i) := h_i w_i$ be the area of the cylinders $\cC_i$, and let $a := \min a(\cC_i)$ (note this depends only on the lattice surface $S$ and not the direction $\theta$), so $$w_i \geq \frac{a}{h_i}$$ Further, we have a constant $d=d(S)$ so that $$\ell_{\cG}(\gamma) \geq d \max_i w_i \geq d \frac{a}{\min_i h_i} = d\frac{a}{|\sin(\theta)|} \frac{1}{\min_i h_i'}.$$ Thus, $$\min_{j} (t_{j+1}- t_j) \geq \min_i h_i' \geq \frac{ad}{|\sin(\theta)| \ell_{\cG}(\gamma)},$$ 
proving our lemma with $C=ad$. Since both $a$ and $d$ depend only
on $S$, the constant $C$ is independent of both $\theta$ and $\gamma$.\end{proof} 

\begin{figure}
    \centering
 \begin{tikzpicture}[scale=1.0]
\draw(0,0)--(3, 0)--(2.4, 1.2)--cycle;
\node[below] at (1.5, 0){\tiny $h'$};

\draw[thick, red](-1, -.5)--(3, 1.5);
\draw[thick, red](1, -1)--(5, 1); 
\node[right] at (2.8, .65){\tiny $h$};
\draw[black](2.4,1.2)--(3, 0);
\draw[thick] (0.4,0) arc (0:27:0.4);
\draw [color=black] (0.4,0.15) node [right] {\tiny $\theta$};
\draw (2.82,-0.09)--(2.73,0.09)--(2.91,0.18);

\end{tikzpicture}
    \caption{Horizontal segment crossing a cylinder. Note that the height $h = h' |\sin(\theta)|$, where $h'$ is the length of a crossing horizontal segment.}
    \label{fig:crossing}
\end{figure}

\paragraph*{\bf Proving Lemma~\ref{lemma:intersections}}
Let
$
d_\theta:=\frac{b_1}{2}|\cot\theta|,
$
where $\theta=\theta(\gamma)$. Since
$\gamma\notin A^{\epsilon,\gamma_0}$, we have
$|\sin\theta|>\epsilon$, and hence
\[
d_\theta
\leq \frac{b_1}{2\epsilon}
\leq \frac{b_2}{2},
\]
by our choice $b_1/b_2\leq\epsilon$.

\paragraph*{\bf Good crossings} Consider an intersection of $\gamma$ with $\gamma_0$ whose horizontal
coordinate $x$ satisfies
\[
b_2+\eta+d_\theta
\leq x\leq
\ell_{\cG}(\gamma_0) +b_2-\eta-d_\theta.
\]
The segment of $\gamma$ crossing the strip through this intersection
lies entirely in the region where $f_0=1/b_1$. Its length in the strip
is $b_1/|\sin\theta|$, and therefore its contribution to
$|\sin\theta|\int_S f_0\,d\sigma_\gamma$ is exactly
\[
|\sin\theta|\frac{1}{b_1}
\frac{b_1}{|\sin\theta|}=1.
\]
Thus these crossings give exactly their contribution to
$\iota_{\cG}(\gamma,\gamma_0)$.

\paragraph*{\bf Exceptional crossings} It remains to control the exceptional crossings: those meeting
$E_-\cup E_+$, together with intersections of $\gamma$ with
$\gamma_0$ lying within distance $\eta+d_\theta$ of one of its
endpoints. Since $\eta<b_2$ and $d_\theta\leq b_2/2$, all such
crossings are contained, near each endpoint, in finitely many
Euclidean sector rectangles of horizontal size bounded by a constant
times $b_2$. The number of these rectangles depends only on $S$.

\paragraph*{\bf Applying spacing estimates} We apply Lemma~\ref{lemma:spacings} separately to the corresponding
horizontal segments in each sector. These segments may be chosen to
issue from the singular endpoint. Since in the direction $\theta$ the
singularity lies on the boundary of a cylinder, it is itself an
intersection point of the collection of saddle connections in direction
$\theta$ with such a segment. Thus the spacing estimate gives, without
an additional endpoint term, that the total number $M$ of exceptional
crossings satisfies
\[
M\leq C_1(S)b_2|\sin\theta|\ell_{\cG}(\gamma).
\]

\paragraph*{\bf Bounding the number of rectangles} Moreover, since $0\leq f_0\leq1/b_1$, the vertical size of each sector
rectangle is bounded by a constant times $b_1$, and the number of such
rectangles is bounded in terms of $S$. Hence each exceptional crossing
contributes at most $C_2(S)$ to
$|\sin\theta|\int_S f_0\,d\sigma_\gamma$. Consequently,
\[
\begin{aligned}
\left|
|\sin\theta|\int_S f_0\,d\sigma_\gamma
-\iota_{\cG}(\gamma,\gamma_0)
\right|
&\leq C_2(S)M\\
&\leq C_3(S)b_2|\sin\theta|\ell_{\cG}(\gamma)\\
&\leq \widetilde C(S)\epsilon\ell_{\cG}(\gamma),
\end{aligned}
\]
where in the last inequality we use $2b_2<\epsilon$ and
$|\sin\theta|\leq1$. This proves \eqref{eq:intersections}.\qed

\subsection{Proving Proposition~\ref{prop:regular}}\label{sec:proofregular}
We now show how to combine Lemmas~\ref{lemma:sectors},
\ref{lemma:nonequidist} and \ref{lemma:intersections} to prove
Proposition~\ref{prop:regular}. Let $\widetilde C=\widetilde C(S)$ be
the constant in Lemma~\ref{lemma:intersections}, and set
$
\delta:={\epsilon}/{(\widetilde C+1)}.
$
We use the function $f_0$ constructed above with parameter $\delta$,
and put
\[
g_0=f_0-\ell_{\cG}(\gamma_0).
\]
By our normalization of $f_0$, the function $g_0$ is continuous and
has mean zero.

Suppose that
\[
\gamma\notin A^{\delta,\gamma_0}\cup I^{\delta,g_0}.
\]
Since $\gamma\notin I^{\delta,g_0}$, we have
\[
\left|\int_S g_0\,d\sigma_\gamma\right|
<\delta\ell_{\cG}(\gamma),
\]
and hence, since $\sigma_\gamma(S)=\ell_{\cG}(\gamma)$,
\[
\left|
\int_S f_0\,d\sigma_\gamma
-\ell_{\cG}(\gamma_0)\ell_{\cG}(\gamma)
\right|
<\delta\ell_{\cG}(\gamma).
\]
On the other hand, since $\gamma\notin A^{\delta,\gamma_0}$,
Lemma~\ref{lemma:intersections} gives
\[
\left|
|\sin\theta(\gamma)|\int_S f_0\,d\sigma_\gamma
-\iota_{\cG}(\gamma,\gamma_0)
\right|
\leq
\widetilde C\delta\ell_{\cG}(\gamma).
\]
Using
\[
\iota_{\cR}(\gamma,\gamma_0)
=
|\sin\theta(\gamma)|
\ell_{\cG}(\gamma)\ell_{\cG}(\gamma_0),
\]
we therefore obtain
\[
\begin{aligned}
|\iota_{\cG}(\gamma,\gamma_0)
-\iota_{\cR}(\gamma,\gamma_0)|
&\leq
\widetilde C\delta\ell_{\cG}(\gamma)
+|\sin\theta(\gamma)|\delta\ell_{\cG}(\gamma)\\
&\leq
(\widetilde C+1)\delta\ell_{\cG}(\gamma)\\
&=
\epsilon\ell_{\cG}(\gamma).
\end{aligned}
\]
Thus
\[
I^{\epsilon,\gamma_0}
\subset
A^{\delta,\gamma_0}\cup I^{\delta,g_0}.
\]
Lemma~\ref{lemma:sectors} gives
\[
N_{\cG}(A^{\delta,\gamma_0},S,L)
\leq c_A\delta L^2,
\]
while Lemma~\ref{lemma:nonequidist} gives
\[
N_{\cG}(I^{\delta,g_0},S,L)=o(L^2).
\]
Since $\delta\leq\epsilon$, this proves
\eqref{eq:regularbound:single} and hence
Proposition~\ref{prop:regular}.

\subsection{The polygonal region}\label{subsec:omega} We now prove that the region $\Omegaone_{S, \cC}$ is a convex, centrally symmetric polygon, via the following elementary geometric lemma. This will be crucial for our explicit computations of our constants in \S\ref{sec:constants:Ngons}

\begin{lemma}\label{lemma:omega} Let $z_0, \ldots, z_{m-1} \in \C^*$, with $z_i/z_j \notin \R$ if $i \neq j$. Then $$\Omega = \left\{z \in \C: \sum_{j=0}^{m-1} |z \wedge z_j| \le 1\right \}$$ is a convex, centrally symmetric $2m$-gon with vertices at $\pm \tau_j z_j$, where $$\tau_j := \left(\sum_{i=0, i \neq j}^{m-1} |z_j \wedge z_i|\right)^{-1}.$$
\end{lemma}

\begin{proof} Central symmetry is clear from the definition of $\Omega$. To see convexity, note that if $z, w \in \Omega$, and $t \in (0, 1)$, $$|((tz) + (1-t)w) \wedge z_j| \le t|z\wedge z_j| + (1-t)|w \wedge z_j|,$$ so \begin{align*} \sum_{j=0}^{m-1} \left|\left(tz+(1-t)w\right) \wedge z_j\right| &\le \sum_{j=0}^{m-1} t|z\wedge z_j| + (1-t)|w \wedge z_j| \\ & \le t\sum_{j=0}^{m-1} |z \wedge z_j|+ (1-t)\sum_{j=0}^{m-1} |w \wedge z_j| \\ & \le t +(1-t) = 1.
\end{align*}
Since $|z\wedge z_j|$ vanishes on the line $L_j = \{t z_j: t \in \R\}$, the extreme points of $\Omega$ must lie on these lines. Direct computation gives the value of $\tau_j$.
\end{proof}

\begin{comment}\subsection*{TO-DO (07 January 2025)}

\begin{itemize}
    \item Flesh out proof of Lemma~\ref{lemma:intersections}
    \item \sout{Confirm Lemma~\ref{lemma:nonequidist} with Dozier \JA{Done, confirmed!}}
    \item \sout{Compute constants for regular $N$-gons. Show that $\Omega_{\cC} = \{z \in \C: \sum_{j=1}^k |z\wedge z_j| \le 1\}$ is always a polygon, and compute explicitly for many examples (regular $n$-gon, billiards in some Veech triangles, $L$-shapes, etc.)}
    \item \sout{Compute first for equilateral triangle, square, and hexagon, and match with CHT constants for triangle and square.}
    \item Check if unfolding of hexagon billiard and unfolding of octahedron are the same translation surface (genus $3$, stratum $\hh(1^{(6)})$) \JA{Not sure if we still want to do this for this paper.}
\end{itemize}
\end{comment}

\section{Asymptotic constants for regular polygons}\label{sec:constants:Ngons}

\paragraph*{\bf Complexity for $N$-gons} \ In this section, we will compute the constants $c_N$ appearing in Theorem~\ref{theorem:billiardcomplexity:Ngons}. We will show below that these constants will be of the following form:

\begin{lemma}\label{lemma:ngonconstants} For $N \geq 5$, the asymptotic complexity constants for billiards in the regular $N$-gon are given by \begin{align}\label{eq:complexity:diagonals} c_{2k} &= \frac 1 3 c_{\cC_{2k}}(S_{2k}) \\ \nonumber c_{2k+1} &= \frac 1 6 c_{\cC_{2k+1}}(S_{2k+1})\end{align} where $c_{\cC_N}(S_N)$ are the constants appearing in Theorem~\ref{theorem:comblengthcounting} (precisely, the right hand side of \eqref{eq:comblengthcounting}) for the lattice translation surface $S_N$ obtained by identifying opposite sides of the regular $N$-gon (for $N$ even), and opposite sides of the doubled regular $N$-gon (for $N$ odd). The system of saddle connections $\cC_N$ on $S_N$ is given by one representative of each set of identified pairs in either case.
\end{lemma}
\paragraph*{\bf Notation} We recall from \S\ref{subsec:lattices}, that for a lattice translation surface $S$ and choice $\cC$ of a filling system of geodesic arcs, the combinatorial counting constant $c_{\cC}(S)$ is given by, for $\cC = \{\gamma_j\}_{j=0}^{m-1} \subset \SC(S)$, $$c_{\cC}(S) = c^{(1)}_{\cG}(S)\vol(\Omega_{\cC}^{(1)}),$$ where $$c^{(1)}_{\cG}(S) = \lim_{R \rightarrow \infty} \frac{N_{\cG}(S^{(1)}, R)}{\pi R^2},$$ and $$\Omega_{\cC}^{(1)} = \left\{z \in \C: \sum_{j=0}^{m-1} \left|z \wedge z_j^{(1)}\right| \le 1\right\},$$ where $$z_j^{(1)} = \int_{\gamma_j} \omega^{(1)} \in \C$$ are the holonomy vectors of the saddle connections $\gamma_j$ on the area $1$ surface $S^{(1)} = (X, \omega^{(1)})$.

\subsection{Billiards, unfolding and covers}\label{subsec:covers} Associated to the billiard $B_N$ in the polygon $P_N$ is the natural \emph{unfolding} to a translation surface $X_N$, of genus $g_N \geq 1$. $X_N$ covers another (primitive) translation surface $S_N$, and we will describe a relationship between counting generalized diagonals for $B_N$ and saddle connections on $S_N$. We let $DB_N$ denote the double of $P_N$ along its boundary which is a flat sphere with $N$ points with angle $2\pi\frac{N-2}{N}$, an $N$-gon \emph{pillowcase}. This pillowcase construction is a standard intermediate step in considering the unfolding of billiards and polyhedra, see, for example~\cite{AthreyaAulicinoHooper}*{\S2} or~\cite{HooperSchwartz}*{\S9}.

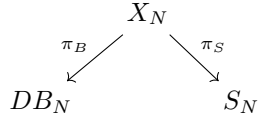
\begin{figure}[ht!]
\begin{tikzcd}[column sep=small]
& X_N \arrow[dl, "\pi_B"'] \arrow[dr, "\pi_S"] & \\
DB_N & & S_N
\end{tikzcd}
\caption{Unfolding the billiard $B_N$ and the associated flat sphere $DB_N$ to a translation surface $X_N$ which covers a primitive lattice surface $S_N$. $\pi_S$ is of degree $N$. For $N=2k$ even, $\pi_B$ is of degree $N/2 = k$, and for $N=2k+1$ odd, $\pi_B$ is of degree $N=2k+1$.}\label{fig:unfoldingNgon}
\end{figure}
\paragraph*{\bf Even-gons} For $N$ even, $DB_N$ is an $N/2$-differential on $\C P^1$ with $N = 2k$ simple poles, with angles $2\pi\frac{N-2}{N} = 2\pi\frac{k-1}{k}$. Taking the degree $k=N/2$ cover of $DB_N$ we obtain the translation surface $X_N$ corresponding to unfolding the billiard $B_N$, which has $N$ singular points with angle $\pi (N-2) = 2\pi\left(\frac{N}{2} - 1\right)$, that is $X_N \in \Omega_{g_N}\left( (N/2-2)^{N}\right)$, where $2g_N - 2 = N(N/2-2)$, so $g_N = N/2(N/2-2)+1 = (N/2-1)^2$. In this case, $X_N$ is a degree $N$ cover of the translation surface $S_N$ obtained by identifying opposite sides of the regular $N$-gon $P_N$. Here, if $N/2 = 2k$, so $N=4k$, $S_N \in \Omega_{k}(2k-2)$, and if $N/2=2k+1$, so $N=4k+2$, $S_N \in \Omega_k(k-1, k-1)$.

\paragraph*{\bf Odd-gons} For $N=2k+1$ odd, $DB_N$ is an $N$-differential on $\C P^1$ with $N$ double poles, that is, points with cone angle $2\pi \frac{N-2}{N}$. Taking the degree $N$ cover of $DB_N$ (fully) ramified over these points, we obtain a surface $X_N$ corresponding to unfolding the billiard $B_N$, which has $N$ cone points of angle $2\pi(N-2)= 2\pi(2k-1)$, so $X_N \in \Omega_{g_N}( (N-3)^N)$ and so has genus $g_N$ satisfying $2g_N-2 = N(2k-2) = (2k+1)(2k-2)$, so $g_N = 2k^2-k = \frac{(N-1)(N-2)}{2}$. In this case, $X_N$ is a degree $2k+1$-cover of the translation surface $S_N$ obtained by identifying opposite sides of the doubled regular $N$-gon $DP_N$, and $S_N \in \Omega_{k}(2k-2)$. 

\paragraph*{\bf Triangles and squares} For $N=3, 4$, the surface $X_N$ is of genus $1$, which needs to be treated somewhat differently. The formula given by Veech~\cite{Veech:eisenstein}*{\S 7} for general $N$ for the affine stabilizer group differs by an index $2$ factor from the actual affine group for $N=3, 4$.

\paragraph*{\bf Generalized diagonals, lifts, and saddle connections} We note that every generalized diagonal for the billiard $B_N$ has two preimages on $DB_N$. Since every singular point on $DB_N$ lifts to a singular point on $X_N$, each generalized diagonal on $B_N$ has $2\deg(\pi_B)$ preimages, all of the same length, and each of different angles, on $X_N$ (so for $N = 2k$ even, $N$ preimages, and for $N$ odd, $2N$ preimages). Moreover, each saddle connection on $S_N$ lifts to $\deg(\pi_S) = N$ distinct \emph{parallel} preimages (of the same length) on $X_N$, since all singularities on $S_N$ lift to singularities on $X_N$. By~\cite{CHT}*{Theorem 1.1}, \begin{equation}\label{eq:regularNgon:sum} \rho(P_N, t) = \sum_{j=0}^{t-1} N_{\cC}(B_N, j),\end{equation} where $N_{\cC}(B_N, j)$ is the number of generalized diagonals on $B_N$ with combinatorial length (that is, number of bounces) at most $j$. If $\gamma$ is a generalized diagonal of $B_N$ with combinatorial length $\ell_{\cC}(\gamma) = t$, it lifts to two saddle connections $\gamma^{\pm}$ on $DB_N$ with combinatorial length $t$ with respect to the system of saddle connections $\mathcal{S}_N$ corresponding to the edges of $P_N$, and any lift $\tilde\gamma^{\pm}$ to $X_N$ has combinatorial length $t$ with respect to the system $\tilde \cC_N$ of saddle connections given by the lifts of all the sides of $P_N$ to $X_N$. The image saddle connection $\pi_S(\tilde\gamma^{\pm})$ on $S_N$ of any lift under $\pi_S$ has combinatorial length $t$ with respect to the system of saddle connections $\cC_N$ on $S_N$ corresponding to the complete set $\{\gamma_0, \ldots, \gamma_{N/2-1}\}$ of representatives of the sets of parallel pairs of sides of $P_N$ (for $N$ even) and to the set of all sides $\{\gamma_0, \ldots, \gamma_{N-1}\}$ of one of the regular $N$-gons in the double regular $N$-gon $DP_N$ (for $N$ odd). Thus, $$N_{\tilde \cC_N}(X_N, t) = \deg(\pi_B) N_{\cC}(DB_N, t) = 2\deg(\pi_B)N_{\cC}(B_N,t) = \deg(\pi_S) N_{\cC_N}(S_N, t),$$ and so \begin{equation}\label{eq:unfolding:counting} N_{\cC}(B_N, t) = \begin{cases}
N_{\cC_N}(S_N, t), N \mbox{ even}\\
\frac{1}{2} N_{\cC_N}(S_N, t), N \mbox{ odd}\\
\end{cases} .\end{equation} Combining \eqref{eq:regularNgon:sum} and \eqref{eq:unfolding:counting} we obtain, as claimed at the start of this section \begin{equation}\label{eq:unfolding:constants} c_N = \lim_{t \rightarrow \infty} \frac{\rho(P_N, t)}{t^3} = \frac{1}{3} \lim_{t \rightarrow \infty} \frac{N_{\cC}(B_N, t)}{t^2} = \begin{cases}
\frac{1}{3} c_{\cC_N}(S_N), N \mbox{ even}\\
\frac{1}{6} c_{\cC_N}(S_N), N \mbox{ odd}\\
\end{cases}.\end{equation}
\begin{figure}[ht!]

\begin{tikzpicture}[scale=2.0]
\draw[thick, red] (0:1)--(45:1)--(90:1)--(135:1)--(180:1);
\draw (180:1)--(225:1)--(270:1)--(315:1)--(0:1);
\filldraw [black] (0:1) circle (.5pt);
\filldraw [black] (45:1) circle (.5pt);
\filldraw [black] (90:1) circle (.5pt);
\filldraw [black] (135:1) circle (.5pt);
\filldraw [black] (180:1) circle (.5pt);
\filldraw [black] (225:1) circle (.5pt);
\filldraw [black] (270:1) circle (.5pt);
\filldraw [black] (315:1) circle (.5pt);

\end{tikzpicture}
\caption{Representatives (in \textcolor{red}{red}) of the saddle
connections in $\cC_8$ in the polygonal presentation of $S_8\in\hh(2)$; each non-highlighted opposite side is identified with the corresponding highlighted side.}
%\caption{The system of saddle connections (in \textcolor{red}{red}) $\cC_8$ on the surface $S_8 \in \hh(2)$.}
\label{fig:cN:even}
\end{figure}
\begin{figure}[ht!]
\begin{tikzpicture}[scale=2.0]

\begin{scope}[shift={(-1.8,0)},rotate=180]
    \draw (0:1)--(360/7:1)--(720/7:1)--(1080/7:1)--(1440/7:1)--(1800/7:1)--(2160/7:1)--cycle;
    \filldraw [black] (0:1) circle (.5pt);
\filldraw [black] (360/7:1) circle (.5pt);
\filldraw [black] (720/7:1) circle (.5pt);
\filldraw [black] (1080/7:1) circle (.5pt);
\filldraw [black] (1440/7:1) circle (.5pt);
\filldraw [black] (1800/7:1) circle (.5pt);
\filldraw [black] (2160/7:1) circle (.5pt);
\end{scope}
\draw[thick, red] (0:1)--(360/7:1)--(720/7:1)--(1080/7:1)--(1440/7:1)--(1800/7:1)--(2160/7:1)--cycle;
    \filldraw [black] (0:1) circle (.5pt);
\filldraw [black] (360/7:1) circle (.5pt);
\filldraw [black] (720/7:1) circle (.5pt);
\filldraw [black] (1080/7:1) circle (.5pt);
\filldraw [black] (1440/7:1) circle (.5pt);
\filldraw [black] (1800/7:1) circle (.5pt);
\filldraw [black] (2160/7:1) circle (.5pt);
\end{tikzpicture}
\caption{Representatives (in \textcolor{red}{red}) of the saddle
connections in $\cC_7$ in the polygonal presentation of
$S_7\in\hh(4)$; the corresponding sides of the other copy are
identified with these.}
%\caption{The system of saddle connections (in \textcolor{red}{red}) $\cC_7$ on the surface $S_7 \in \hh(4)$.}
\label{fig:cN:odd}
\end{figure}

\subsection{Areas, lengths, and normalized intersections}\label{subsec:areas} We now turn to computing the constants $c_{\cC_N}(S_N)$. We first discuss the general relationship between geometric and combinatorial asymptotic constants, and area normalizations.  In this discussion we will use $S = (X, [\omega])$ to denote a translation surface without a fixed area normalization, and, as before, $S^{(a)} = (X, \omega^{(a)})$ to denote the area $a$ realization, so $[\omega]$ denotes the equivalence class of differentials up to (positive real) scaling, and $$\area(S^{(a)}) = \frac i 2 \int_{X} \omega^{(a)} \wedge \bar \omega^{(a)} = a.$$ Let $\Sone := (X, \omegaone)$ be an area $1$ scaling of a translation surface, and for a filling system of geodesic arcs $\cC = \{\gamma_0, \ldots, \gamma_{m-1}\}$ on $S$, and any saddle connection $\gamma \in \SC(S)$, recall that the \emph{regularized intersection} is given by $$\ir(\gamma, \cC) = \sum_{j=0}^{m-1} \left|\zone_{\gamma} \wedge \zone_j\right|,$$ where $$\zone_{\gamma} = \int_{\gamma} \omegaone, \zone_j = \int_{\gamma_j} \omegaone.$$ For $a>0$, consider the area $a$ surface $$\Sa = (X, \omegaa) := (X, \sqrt a \omegaone).$$ Recall that for any saddle connection $\gamma \in \SC(S)$, $$\za_{\gamma} = \int_{\gamma} \omegaa= \sqrt a \zone_{\gamma},$$ so $$\ir(\gamma, \cC) =  \frac{1}{a}\sum_{j=0}^{m-1} |\za_{\gamma} \wedge \za_j|$$ is well defined independent of $a>0$. Recall that $$\Ng(\Sone, R) = \#\{\gamma \in \SC(S): |\zone_{\gamma}| \le R\},$$ and \begin{equation}\label{eq:cgone} \cone_{\cG} = \lim_{R\rightarrow \infty} \frac{\Ng(\Sone, R)}{\pi R^2}. \end{equation} This differs from the normalization used in~\cite{Veech:eisenstein}*{Theorem 1.5} by a factor of $\pi$, which will be important in our computations moving forward. Defining \begin{equation}\label{eq:cga} \ca_{\cG} = \lim_{R\rightarrow \infty} \frac{\Ng(\Sa, R)}{\pi R^2}, \end{equation} we have \begin{align*}
    \ca_{\cG} &= \lim_{R\rightarrow \infty} \frac{\Ng(\Sa, R)}{\pi R^2} \\ &=  \lim_{R\rightarrow \infty} \frac{\Ng(\Sone, R/\sqrt a)}{\pi (R/\sqrt a)^2 a} \\ &= \cone_{\cG}/a.
\end{align*}
Therefore, we have \begin{equation}\label{eq:cgonea} \cone_{\cG} = a \ca_{\cG}, \end{equation} so $a \ca_{\cG}$ is well-defined independent of $a$. Of course, our combinatorial length counting is independent of area, and to reconcile this with our description of the asymptotic constant using geometric counting and the region $\Omega_{\cC}$, we have to understand how this region changes as we scale the surface. We define, for the unit area surface $\Sone$, \begin{equation}\label{eq:Omegaone}
    \Omegaone_{S, \cC} := \left\{z \in \C: \sum_{j=0}^{m-1} \left|z \wedge \zone_j\right| \le 1\right\}
\end{equation}
and for the scaled surface $\Sa$,
 \begin{equation}\label{eq:Omegaa}
    \Omegaa_{\cC} := \left\{z \in \C: \frac 1 a \sum_{j=0}^{m-1} |z \wedge \za_j| \le 1\right\} = \sqrt a \Omegaone_{S, \cC}
\end{equation}
Thus,  \begin{equation}\label{eq:Omegavolume}
    |\Omegaa_{\cC}|  = a|\Omegaone_{S, \cC}|.
\end{equation}
By Theorem~\ref{theorem:comblengthcounting} \begin{align*} \lim_{t \rightarrow \infty} \frac{\Nc(S, t)}{t^2} &= \cone_{\cG} |\Omegaone_{S, \cC}| \\ &= a\ca_{\cG} |\Omegaone_{S, \cC}| \\ &= \ca_{\cG} |\Omegaa_{\cC}|.\end{align*} Thus, we've shown that $\ca_{\cG} |\Omegaa_{\cC}|$ is independent of $a$. For our computations, it will be convenient to work with the auxiliary set \begin{equation}\label{eq:omegahat} \hat\Omega^{(a)}_{\cC} := \left\{z \in \C: \sum_{j=0}^{m-1} \left|z \wedge \za_j\right| \le 1\right\}. 
\end{equation} We have $$|\hat\Omega^{(a)}_{\cC}|  = \frac{1}{a}|\Omegaone_{S, \cC}|,$$ so \begin{align}\label{eq:ncomba}
    \lim_{t \rightarrow \infty} \frac{\Nc(S, t)}{t^2} &= \cone_{\cG} a|\hat\Omega^{(a)}_{\cC}| \\ \nonumber &= a^2 \ca_{\cG} |\hat\Omega^{(a)}_{\cC}| , \end{align} with the right-hand side being independent of $a$.

\paragraph*{\bf Inscribed $N$-gons} We will treat the cases $N=4k$, $N=4k+2$, and $N=2k+1$ all separately, as there are important subtleties which differ across these, due to the different geometries of the associated surfaces $S_N$. We fix some notation which we will use across all of these cases. We write $e(x) := e^{2\pi i x}$, and for $x, N \in \R$, we write $$e_N(x) := e^{2\pi i x/N}.$$ Note $e_1(x) = e(x)$. Let $P_N$ be the regular $N$-gon inscribed in the unit circle $\{z \in \C: |z|=1\}$ with vertices at $$\zeta_{N, j} = e_N(j), \ j=0, \ldots N-1.$$ The area of $P_N$ is given by $$a_N := \area(P_N) = \frac{N}{2} \sin(2\pi/N).$$
The sides of $P_N$ are given by the differences $$z_{N,j} = \zeta_{N, j+1} - \zeta_{N, j} = e_N(j+1)-e_N(j),$$ all of which have the same length, $$r_N = |z_{N,j}| = |e_N(j+1)-e_N(j)| = \sqrt{2(1- \cos(2\pi/N))}.$$

\subsection{Even-gons}\label{subsec:even} We start by computing these asymptotics for $N$-gons with $N$ even. We will first discuss some general properties, then specify to the cases $N=4k$ and $N=4k+2$.
\paragraph*{\bf The region $\Omega^{(a_N)}_{\cC_N}$, $N$ even} We work with the translation surface $S_N^{(a_N)}$ of area $a_N$ obtained by identifying opposite sides of $P_N$, and the system of saddle connections $\cC_N$ given by this collection of $N/2$ paired sides, with holonomy vectors $z_{N,j}, j=0, \ldots, N/2-1$. We have $$\hat\Omega_{\cC_N}^{(a_N)} = \left\{z \in \C: \sum_{j=0}^{N/2-1} |z \wedge z_{N, j}| \le 1\right\}.$$ \begin{lemma}\label{lemma:omega:even} For $N$ even, the region $\hat\Omega_{\cC_N}^{(a_N)}$ is a regular $N$-gon inscribed in the circle of radius $\frac{1}{2} \sec(\pi/N)$ centered at $0$, with area \begin{equation}\label{eq:omega:even}\left|\hat\Omega_{\cC_N}^{(a_N)}\right| = \frac{N}{4} \tan(\pi/N).\end{equation}   
\end{lemma}
\begin{proof} By Lemma~\ref{lemma:omega}, $\Omega_{\cC}:= \{z \in \C: \sum_{j=0}^{m-1} |z \wedge z_j| \le 1\}$ is a convex, centrally symmetric polygon with $2m$ vertices, with two each along the lines $L_j = \{tz_j\}_{t \in \R}, j =0, \ldots m-1.$ In our setting, since we have $m=N/2$, and the defining $z_{N,j}$ are at equally spaced angles, this implies that our region $\hat\Omega_{\cC_N}^{(a_N)}$ is a regular $N = 2m$-gon, and a direct computation gives the radius.
\end{proof}

\paragraph*{\bf Cusps and counting} Having computed the volume $\left|\hat\Omega_{\cC_N}^{(a_N)}\right|$, it remains for us to compute $\cone_{\cG}(S_N)$. We will, in fact, compute the geometric counting asymptotic constant $c_{\cG}^{(a_N)}(S_N)$ using results of Veech~\cites{Veech:eisenstein, Veech:polygon}. Let $\Gamma_N := \Gamma(S_N)$ denote the Veech group of $S_N$, which, for $N$ even, is given by $$\Gamma_{N} = \Delta^+(N/2, \infty, \infty),$$ where $\Delta^+(N/2, \infty, \infty)$ is the orientation preserving subgroup of the Hecke triangle group $\Delta(N/2, \infty, \infty)$, which has 
$\covol(\Gamma_N)
=2\left(\pi-\frac{2\pi}{N}\right)
=2\pi\left(\frac{N-2}{N}\right)$.
Veech~\cite{Veech:eisenstein} showed that every saddle connection $\gamma \in \SC(S_N)$ is either of the form $$\gamma = g \gamma_0, g \in \Gamma_N,$$ where $\gamma_0$ is horizontal, or of the form $$\gamma = g \gamma_{\pi/N}, g \in \Gamma_N,$$ where $\gamma_{\pi/N}$ makes angle $\pi/N$ with the horizontal. We denote the (finite) sets of horizontal and angle $\pi/N$ saddle connections by $\SC_0^{\dagger}(S_N)$ and $\SC_{\pi/N}^{\dagger}(S_N)$ respectively, and let $\SC_0(S_N) := \Gamma_N \cdot \SC_0^{\dagger}(S_N)$ and $\SC_{\pi/N}(S_N) := \Gamma_N \cdot \SC_{\pi/N}^{\dagger}(S_N)$. We have $$\SC(S_N) = \SC_0(S_N) \sqcup \SC_{\pi/N}(S_N).$$ Thus, setting \begin{align*}
N^{(0)}_{\cG}(S_N^{(a_N)}, t) &= \#\{\gamma \in\SC_0(S_N): |z_{\gamma}^{(a_N)}| \le t\}    \\
N^{(\pi/N)}_{\cG}(S_N^{(a_N)}, t) &= \#\{\gamma \in\SC_{\pi/N}(S_N): |z_{\gamma}^{(a_N)}| \le t\},  
\end{align*}
we have \begin{equation}\label{eq:even:cusp} c_{\cG}^{(a_N)}(S_N) = c^{(a_N)}_0(S_N) + c_{\pi/N}^{(a_N)}(S_N),\end{equation} where $$c_{\ast}^{(a_N)}(S_N) = \lim_{t \rightarrow \infty} \frac{N^{\ast}_{\cG}(S_N^{(a_N)}, t)}{\pi t^2}.$$ Veech~\cite{Veech:eisenstein} showed how to compute the constants $\tilde c_{\ast}^{(a_N)}(S_N) = \pi c_{\ast}^{(a_N)}(S_N)$, where $$\tilde c_{\ast}^{(a_N)}(S_N) = \lim_{t \rightarrow \infty} \frac{N^{\ast}_{\cG}(S_N^{(a_N)}, t)}{t^2}.$$ Veech~\cite{Veech:eisenstein}*{Proposition 3.11, \S 5} showed \begin{equation}\label{eq:veech:even}
    \tilde c_{\ast}^{(a_N)} = \frac{1}{\covol(\Gamma_N)} 2\cot(\pi/N) \left( \sum_{\gamma \in\SC_{\ast}^{\dagger}(S_N)} \frac{1}{\left|z^{(a_N)}_{\gamma}\right|^2} \right)
\end{equation}

\subsubsection{$N=4k$-gons}\label{subsec:4kgons} We now specify to the case $N=4k$. Here, $$a_N = \area(P_N) = N/2 \sin(2\pi/N) = 2k \sin(\pi/2k),$$ and $$\left|\hat\Omega_{\cC_N}^{(a_N)}\right| = \frac{N}{4} \tan(\pi/N) = k \tan(\pi/4k).$$

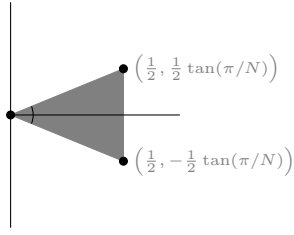
\begin{figure}[ht!]
\begin{tikzpicture}[scale=3.0]

\filldraw[gray](0,0)--(1/2, .205)node[right]{\tiny $\left(\frac 1 2, \frac 1 2 \tan(\pi/N)\right)$}--(1/2, -.205)node[right]{\tiny $\left(\frac 1 2, -\frac 1 2 \tan(\pi/N)\right)$}--cycle;
\draw (0,0)--(.75,0);
\draw(0, -.5)--(0, .5);
\filldraw [black] (1/2, .205) circle (.5pt);
\filldraw [black] (1/2, -.205) circle (.5pt);
\filldraw [black] (0,0) circle (.5pt);
\draw[thin] (0:.1) arc (0:22.5:.1);
\draw[thin] (0:.1) arc (0:-22.5:.1);
\end{tikzpicture}\caption{A piece of $\hat\Omega_{\cC_N}^{(a_N)}$, $N=8$.}\label{fig:omega4k}
\end{figure}

\subsubsection{The horizontal cusp, $N=4k$}\label{subsec:cusp:horizontal:4k} We now show how to use \eqref{eq:veech:even} to compute $c_{0}^{(a_N)} = \frac{1}{\pi} \tilde c_{0}^{(a_N)}$. We start by computing the set $\SC_0^{\dagger}(S_N).$ In the model $S_N^{(a_N)}$, obtained by identifying opposite sides of the regular $N$-gon $P_N$ inscribed in the unit circle, there is a unique singular point corresponding to the vertices $e_N(n), n=0, \ldots, N-1$ (which are all identified on $S_N$), and there are horizontal saddle connections (see Figure~\ref{fig:4k:horizontal}) $\gamma_j, j=-(k-1), \ldots, k-1$, with holonomy vectors $$z^{(a_N)}_j = \int_{\gamma_j} \omega^{(a_N)} = e_N(j)- e_N(2k-j).$$ For $j=0$, the saddle connection $\gamma_0$ has length $2$ on $S_N^{(a_N)}$, and for $j = 1, \ldots, k-1$, the two saddle connections $\gamma_{\pm j}$ each have length \begin{align*}\left|z^{(a_N)}_j\right| &= \left|e_N(j)- e_N(2k-j)\right| \\ &= \left|e_N(j) + e_N(-j)\right| \\&= 2\cos\left(\frac{2\pi j}{N}\right) \\ &= 2\cos\left(\frac{j}{2k}\pi\right).\end{align*}
\begin{figure}[ht!]
\begin{tikzpicture}[scale=3.0]
\draw (0:1)node[right]{\tiny $e_{12}(0)$}--(30:1)node[right]{\tiny $e_{12}(1)$}--(60:1)node[right]{\tiny $e_{12}(2)$}--(90:1)node[above]{\tiny $e_{12}(3)$}--(120:1)node[left]{\tiny $e_{12}(4)$}--(150:1)node[left]{\tiny $e_{12}(5)$}--(180:1)node[left]{\tiny $e_{12}(6)$}--(210:1)node[left]{\tiny $e_{12}(7)$}--(240:1)node[left]{\tiny $e_{12}(8)$}--(270:1)node[below]{\tiny $e_{12}(9)$}--(300:1)node[right]{\tiny $e_{12}(10)$}--(330:1)node[right]{\tiny $e_{12}(11)$}--cycle;\draw[red](0:1)--(180:1);
\draw[blue](30:1)--(150:1);
\draw[blue](60:1)--(120:1);
\node(0,0.2){\tiny $S_{12}$};
\draw[blue](-30:1)--(-150:1);
\draw[blue](-60:1)--(-120:1);
\filldraw [black] (0:1) circle (.5pt);
\filldraw [black] (30:1) circle (.5pt);
\filldraw [black] (60:1) circle (.5pt);
\filldraw [black] (90:1) circle (.5pt);
\filldraw [black] (120:1) circle (.5pt);
\filldraw [black] (150:1) circle (.5pt);
\filldraw [black] (180:1) circle (.5pt);
\filldraw [black] (210:1) circle (.5pt);
\filldraw [black] (240:1) circle (.5pt);
\filldraw [black] (270:1) circle (.5pt);
\filldraw [black] (300:1) circle (.5pt);
\filldraw [black] (330:1) circle (.5pt);
\end{tikzpicture}
\caption{Horizontal saddle connections on $S_{12} \in \hh(4)$. The \textcolor{red}{red} saddle connection is of length $2$ and the \textcolor{blue}{blue} saddle connections occur in pairs of the same length. Note $e_N(-j) = e_N(N-j)$.}\label{fig:4k:horizontal}
\end{figure}
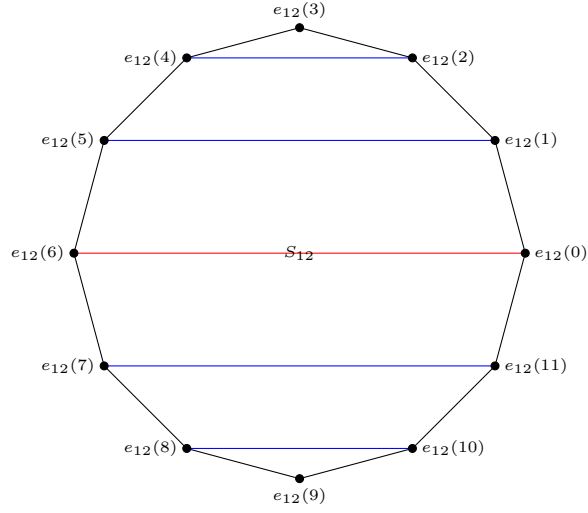
Applying \eqref{eq:veech:even}, we have \begin{align}\label{eq:veech:even:horizontal}
    c_{0}^{(a_N)} &= \frac{1}{\pi}\frac{1}{\covol(\Gamma_N)} 2\cot(\pi/N) \left( \sum_{\gamma \in\SC_{0}^{\dagger}(S_N)} \frac{1}{\left|z^{(a_N)}_{\gamma}\right|^2} \right)\\ \nonumber &=  \frac{1}{\pi^2}\frac{N}{(N-2)} \cot(\pi/N)  \left( \frac{1}{2^2} + 2 \sum_{j=1}^{k-1} \frac{1}{4 \cos^2\left(\frac{j}{2k} \pi\right)} \right).
\end{align}
To compute this, we need to compute the sum \begin{equation}\label{eq:sigma0k:def} \Sigma_0(k) := \sum_{j=1}^{k-1} \frac{1}{\cos^2\left(\frac{j}{2k} \pi\right)}.\end{equation} For this, we will use the following beautiful formula, which we refer to as the \emph{fundamental identity}, which has several different interesting proofs (see, for example~\cite{Devriendt}, ~\cite{Hofbauer},~\cite{Kortram}) including one by Veech~\cite{Veech:eisenstein}*{Lemma 6.3}. 

\begin{lemma}\label{lemma:veech:identity} For $m>1$, \begin{equation}\label{eq:veech:identity}\sum_{j=1}^{m-1} \frac{1}{\sin^2(\pi j/m)} = \frac{m^2-1}{3}.\end{equation}
\end{lemma}

\paragraph*{\bf Applying the fundamental identity} Using \eqref{eq:veech:identity}, we will show \begin{equation}\label{eq:sigma0k} \Sigma_0(k) = \frac{2}{3}(k^2-1).
\end{equation}
Since $\cos(\theta) = \sin(\pi/2-\theta)$, we have \begin{align*}
    \Sigma_0(k) &= \sum_{j=1}^{k-1} \frac{1}{\cos^2(\pi j/2k)}\\
    &= \sum_{t=1}^{k-1} \frac{1}{\cos^2(\pi (k-t)/2k)} \mbox{, putting } t=k-j\\
    &= \sum_{t=1}^{k-1} \frac{1}{\sin^2(\pi t/2k)}.
\end{align*}
Since $\sin^2(\pi-\theta) = \sin^2(\theta),$ $$\sin^2(\pi t/2k) = \sin^2(\pi -\pi t/2k) = \sin^2\left(\frac{2k-t}{2k} \pi\right).$$ 
Thus
\begin{align*}
2\Sigma_0(k)+1
&=1+\sum_{t=1}^{k-1}\frac{2}{\sin^2(\pi t/2k)}\\
&=\frac{1}{\sin^2(\pi k/2k)}
+\sum_{t=1}^{k-1}\frac{1}{\sin^2(\pi t/2k)}
+\sum_{t=k+1}^{2k-1}\frac{1}{\sin^2(\pi t/2k)}\\
&=\sum_{t=1}^{2k-1}\frac{1}{\sin^2(\pi t/2k)}\\
&=\frac{1}{3}\left((2k)^2-1\right).
\end{align*}
%Thus \begin{align*}
%    2\Sigma_0(k)+1 &= 1+\sum_{t=1}^{k-1} \frac{2}{\sin^2(\pi t/2k)}\\
%    &=  1+ \sum_{t=1}^{k-1} \left(\frac{1}{\sin^2(\pi (k-t)/2k)}+\frac{1}{\sin^2\left(\frac{2k-t}{2k} \pi\right)} \right) \\
%    &= \frac{1}{\sin^2(\pi k/2k)} + \sum_{t=1}^{k-1} \frac{1}{\sin^2(\pi t/2k)} + \sum_{t=k+1}^{2k-1} \frac{1}{\sin^2(\pi t/2k)} \\
%    &= \sum_{t=1}^{2k-1} \frac{1}{\sin^2(\pi t/2k)} \\
%    & = \frac{1}{3} \left( (2k)^2 -1 \right),
%\end{align*}
where in the last equality we are using \eqref{eq:veech:identity} with $m=2k$. This proves \eqref{eq:sigma0k}. Plugging this into \eqref{eq:veech:even:horizontal}, we obtain \begin{align}\label{eq:veech:even:horizontal:final} c_{0}^{(a_N)}  &=  \frac{1}{\pi^2}\frac{N}{(N-2)} \cot(\pi/N)  \left( \frac{1}{2^2} + 2 \sum_{j=1}^{k-1} \frac{1}{4 \cos^2(\pi j/2k)} \right)\\ \nonumber &= \frac{1}{\pi^2}\frac{N}{(N-2)} \cot(\pi/N)  \left( \frac{1}{4} + \frac{1}{2} \Sigma_0(k) \right) \\ \nonumber &= \frac{1}{\pi^2}\frac{N}{(N-2)} \cot(\pi/N)  \left( \frac{1}{4} + \frac{k^2-1}{3} \right).
\end{align}
\subsubsection{The $\pi/N$-cusp, $N=4k$.}\label{subsec:cusp:piN:4k} We now turn to the \emph{other} cusp of $\Gamma_N$ which corresponds to the $\Gamma_N$-orbit $\SC_{\pi/N}(S_N) = \Gamma_N \cdot \SC_{\pi/N}^{\dagger}(S_N)$ of the set $\SC_{\pi/N}^{\dagger}(S_N)$ of saddle connections at angle $\pi/N$ with the horizontal. The set $\SC_{\pi/N}^{\dagger}(S_N)$ consists of saddle connections $\gamma_j$ connecting $e_N(j)$ to $e_N(2k+1-j)$, $j=-(k-1), \ldots, k$, see Figure~\ref{fig:4k:piN}. The $j=k$ and $j=-(k-1)$ cases in fact correspond to one saddle connection, corresponding to parallel sides of $S_N$, and for $j=1, \ldots, k-1$, we have two copies of each saddle connection. The lengths of these saddle connections are given by \begin{align*} \left|z_j^{(a_N)}\right| &= |e_{4k}(2k+1-j)- e_{4k}(j)| \\ &=|e_{8k}(4k+2 -2j) - e_{8k}(2j)|\\ &=|e_{8k}(4k+1 -2j) - e_{8k}(2j-1)| \mbox{ }(\mbox {multiplying by } |e_{8k}(-1)|)\\ &=|(-1)(e_{8k}(1 -2j) + e_{8k}(2j-1))| \mbox{ }(\mbox {using } e_{8k}(4k) = e(1/2) = -1) \\
&= 2\cos \left( \frac{2j-1}{4k} \pi \right).
\end{align*}
For $j=k$, this gives
\[
\left|z_k^{(a_N)}\right|
=
2\cos\left(\frac{2k-1}{4k}\pi\right)
=
2\sin\left(\frac{\pi}{N}\right)
=
r_N.
\]
\begin{figure}[ht!]
\begin{tikzpicture}[scale=3.0]
\draw (0:1)node[right]{\tiny $e_{12}(0)$}--(30:1)node[right]{\tiny $e_{12}(1)$}--(60:1)node[right]{\tiny $e_{12}(2)$}--(90:1)node[above]{\tiny $e_{12}(3)$}--(120:1)node[left]{\tiny $e_{12}(4)$}--(150:1)node[left]{\tiny $e_{12}(5)$}--(180:1)node[left]{\tiny $e_{12}(6)$}--(210:1)node[left]{\tiny $e_{12}(7)$}--(240:1)node[left]{\tiny $e_{12}(8)$}--(270:1)node[below]{\tiny $e_{12}(9)$}--(300:1)node[right]{\tiny $e_{12}(10)$}--(330:1)node[right]{\tiny $e_{12}(11)$}--cycle;
\draw[blue](0:1)--(210:1);
\draw[blue](-30:1)--(240:1);
\draw[blue](180:1)--(30:1);
\node(0,0.2){\tiny $S_{12}$};
\draw[blue](150:1)--(60:1);
\draw[thick, red](-60:1)--(-90:1);
\draw[thick, red](120:1)--(90:1);
\filldraw [black] (0:1) circle (.5pt);
\filldraw [black] (30:1) circle (.5pt);
\filldraw [black] (60:1) circle (.5pt);
\filldraw [black] (90:1) circle (.5pt);
\filldraw [black] (120:1) circle (.5pt);
\filldraw [black] (150:1) circle (.5pt);
\filldraw [black] (180:1) circle (.5pt);
\filldraw [black] (210:1) circle (.5pt);
\filldraw [black] (240:1) circle (.5pt);
\filldraw [black] (270:1) circle (.5pt);
\filldraw [black] (300:1) circle (.5pt);
\filldraw [black] (330:1) circle (.5pt);
\end{tikzpicture}
\caption{Saddle connections of angle $\pi/12$ on $S_{12} \in \hh(4)$. The \textcolor{red}{red} saddle connection corresponds to a pair of sides which are identified, and is of length $r_{12} = |e_{12}(1)-1| = \sqrt{2(1- \cos(2\pi/12))}$ and the \textcolor{blue}{blue} saddle connections occur in pairs of the same length.}\label{fig:4k:piN}
\end{figure}
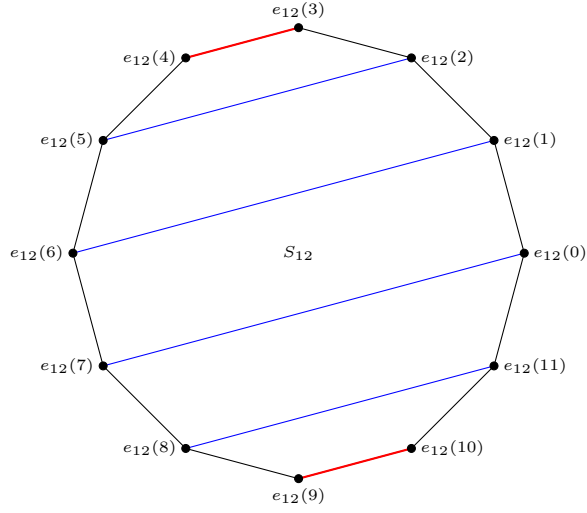
Following Veech, we have:
\begin{align}\label{eq:veech:even:piN}
    c_{\pi/N}^{(a_N)} &= \frac{1}{\pi}\frac{1}{\covol(\Gamma_N)} 2\cot(\pi/N) \left( \sum_{\gamma \in\SC_{\pi/N}^{\dagger}(S_N)} \frac{1}{\left|z^{(a_N)}_{\gamma}\right|^2} \right)\\ \nonumber &=  \frac{1}{\pi^2}\frac{N}{(N-2)} \cot(\pi/N)  \left( \frac{1}{4\cos^2\left(\frac{2k-1}{4k}\pi\right)} + 2 \sum_{j=1}^{k-1} \frac{1}{4 \cos^2\left(\frac{2j-1}{4k}\pi\right)} \right) \\ \nonumber &= \frac{1}{\pi^2}\frac{N}{(N-2)} \cot(\pi/N)  \left( \frac 1 2 \sum_{j=1}^{k} \frac{1}{ \cos^2\left(\frac{2j-1}{4k}\pi\right)} - \frac{1}{4\cos^2\left(\frac{2k-1}{4k}\pi\right)} \right) 
\end{align}
We thus need to evaluate the sum \begin{equation}\label{eq:sigmapiNk:def} \Sigma_{\pi/N}(k) := \sum_{j=1}^{k} \frac{1}{\cos^2\left(\frac{2j-1}{4k}\pi\right)}.\end{equation} We will show, using \eqref{eq:veech:identity}, that \begin{equation}\label{eq:sigmapiNk} \Sigma_{\pi/N}(k)= 2k^2.    
\end{equation}
Putting $t=k+1-j$, note that as $j$ ranges from $1$ to $k$, $t$ ranges from $k$ to $1$, and \begin{align*} \cos^2\left(\frac{2j-1}{4k}\pi\right) &= \cos^2\left(\frac{2(k+1-t)-1}{4k}\pi\right)\\ &= \cos^2\left(\frac{\pi}{2} + \frac{1-2t}{4k}\pi\right) \\&= \sin^2\left(\frac{2t-1}{4k}\pi\right),
\end{align*}
so $$\Sigma_{\pi/N}(k) = \sum_{t=1}^{k} \frac{1}{\sin^2\left(\frac{2t-1}{4k}\pi\right)}.$$ Applying \eqref{eq:veech:identity} with $m=4k$, we have $$\sum_{j=1}^{4k-1} \frac{1}{\sin^2\left(\frac{j}{4k}\pi\right)} = \frac{1}{3} ((4k)^2-1),$$ and by splitting this into odd and even terms, we obtain \begin{align*} \sum_{j=1}^{4k-1} \frac{1}{\sin^2\left(\frac{j}{4k}\pi\right)} &= \sum_{t=1}^{2k} \frac{1}{\sin^2\left(\frac{2t-1}{4k}\pi\right)} + \sum_{t=1}^{2k-1} \frac{1}{\sin^2\left(\frac{2t}{4k}\pi\right)} \\& = \sum_{t=1}^{2k} \frac{1}{\sin^2\left(\frac{2t-1}{4k}\pi\right)} + \sum_{t=1}^{2k-1} \frac{1}{\sin^2\left(\frac{t}{2k}\pi\right)}.
\end{align*}
Applying \eqref{eq:veech:identity} with $m=2k$, $$\sum_{t=1}^{2k-1} \frac{1}{\sin^2\left(\frac{t}{2k}\pi\right)} = \frac 1 3 ( (2k)^2-1),$$ so \begin{align*} \sum_{t=1}^{2k} \frac{1}{\sin^2\left(\frac{2t-1}{4k}\pi\right)} &= \frac{1}{3} ((4k)^2-1) - \frac 1 3 ( (2k)^2-1)\\ &= 4k^2.\end{align*}

On the other hand, the terms in this sum occur in equal pairs.
Indeed, if $s=2k+1-t$, then
\[
\sin^2\left(\frac{2s-1}{4k}\pi\right)
=
\sin^2\left(\pi-\frac{2t-1}{4k}\pi\right)
=
\sin^2\left(\frac{2t-1}{4k}\pi\right).
\]
Thus
\[
\sum_{t=1}^{2k}
\frac{1}{\sin^2\left(\frac{2t-1}{4k}\pi\right)}
=
2\sum_{t=1}^{k}
\frac{1}{\sin^2\left(\frac{2t-1}{4k}\pi\right)}
=
2\Sigma_{\pi/N}(k).
\]
Since the left-hand side is $4k^2$, we obtain
\[
\Sigma_{\pi/N}(k)=2k^2,
\]
as claimed. Therefore,

\begin{align}\label{eq:veech:even:piN:final}
c_{\pi/N}^{(a_N)}
&=
\frac{1}{\pi^2}\frac{N}{N-2}\cot(\pi/N)
\left(
\frac{1}{2}\Sigma_{\pi/N}(k)
-
\frac{1}{4\cos^2\left(\frac{2k-1}{4k}\pi\right)}
\right)
\\ \nonumber
&=
\frac{1}{\pi^2}\frac{N}{N-2}\cot(\pi/N)
\left(
k^2-
\frac{1}{4\cos^2\left(\frac{2k-1}{4k}\pi\right)}
\right).
\end{align}
%We now claim $$\Sigma_{\pi/N}(k) = \sum_{t=1}^{k} \frac{1}{\sin^2\left(\frac{2t-1}{4k}\pi\right)} = \frac 1 2 \sum_{t=1}^{2k} \frac{1}{\cos^2\left(\frac{2t-1}{4k}\pi\right)}.$$ To see this, note for $s=2k+1-t$, if $t$ ranges from $1$ to $k$, $s$ ranges from $2k$ to $k+1$, and \begin{align*} \sin^2 \left( \frac{2t-1}{4k} \pi \right) &= \sin^2 \left( \frac{2(2k+1-s)-1}{4k} \pi \right) \\ &= \sin^2 \left( \pi + \frac{1-2s}{4k} \pi \right) \\ &= \sin^2 \left( \frac{2s-1}{4k} \pi \right).\end{align*}
%Thus, we have \begin{align}\label{eq:veech:even:piN:final} c_{\pi/N}^{(a_N)}  &=  \frac{1}{\pi^2}\frac{N}{(N-2)} \cot(\pi/N)  \left( \frac{1}{2}\Sigma_{\pi/N}(k) - \frac{1}{4 \cos^2\left(\frac{2k-1}{4k}\pi\right)} \right)\\ \nonumber &= \frac{1}{\pi^2}\frac{N}{(N-2)} \cot(\pi/N)  \left( k^2 - \frac{1}{4 \cos^2\left(\frac{2k-1}{4k}\pi\right)} \right).
%\end{align}
    
\subsubsection{Combinatorial constants and $N=4k \rightarrow \infty$ asymptotics}\label{subsec:4k:final} Combining \eqref{eq:ncomba}, \eqref{eq:omega:even}, \eqref{eq:even:cusp}, \eqref{eq:veech:even:horizontal:final}, and \eqref{eq:veech:even:piN:final}, we have for $N=4k$, \begin{align}\label{eq:veech:4k:final}
    c_{\cC_{N}}(S_{N})  &= a_{N}^2 c^{(a_N)}_{\cG} \left|\hat\Omega^{(a_N)}_{\cC}\right| \\ \nonumber &= \frac{N^2}{4} \sin^2\left(\frac{2\pi}{N}\right) \frac{N}{4} \tan\left(\frac{\pi}{N}\right) \left(c^{(a_N)}_0(S_N) + c_{\pi/N}^{(a_N)}(S_N) \right) \\ \nonumber &= \frac{N^3}{16} \sin^2\left(\frac{2\pi}{N}\right) \tan\left(\frac{\pi}{N}\right) \frac{1}{\pi^2} \cot\left(\frac{\pi}{N}\right) \frac{N}{N-2} \left( \frac 1 4 + \frac 1 3 (k^2-1) + k^2 - \frac{1}{4\cos^2\left( \frac{2k-1}{4k} \pi \right)} \right) \\ \nonumber &= \frac{N^4 \sin^2\left(\frac{2\pi}{N}\right)}{16\pi^2(N-2)} \left( \frac{1}{12} N^2 - \frac{1}{4\sin^2\left(\frac{\pi}{N}\right)} - \frac{1}{12}\right),\end{align}
where in the last line we are using $$\cos^2\left( \frac{2k-1}{4k} \pi \right) = \sin^2\left(\frac{\pi}{4k}\right) = \sin^2\left(\frac{\pi}{N}\right).$$ Applying \eqref{eq:complexity:diagonals}, we have, for $N=4k$, \begin{equation}\label{eq:4k:complexity} c_N = \frac{1}{3} c_{\cC_{N}}(S_{N}) =  \frac{N^4 \sin^2\left(\frac{2\pi}{N}\right)}{48\pi^2(N-2)} \left( \frac{1}{12} N^2 - \frac{1}{4\sin^2\left(\frac{\pi}{N}\right)} - \frac{1}{12}\right).\end{equation}
We note that as $k \rightarrow \infty$, $$\sin^2\left(\frac{2\pi}{N}\right) \sim \frac{4\pi^2}{N^2}, \sin^2\left(\frac{\pi}{N}\right) \sim \frac{\pi^2}{N^2},$$ so \begin{align*} c_{4k} &=  \frac{N^4 \sin^2\left(\frac{2\pi}{N}\right)}{48\pi^2(N-2)} \left( \frac{1}{12} N^2 - \frac{1}{4\sin^2\left(\frac{\pi}{N}\right)} \right) \\ &\sim  \frac{N^2 4\pi^2}{48\pi^2(N-2)} \left( N^2\left(\frac{1}{12}- \frac{1}{4\pi^2} \right) \right) \\ &\sim \frac{N}{12}\left( N^2\left(\frac{1}{12}- \frac{1}{4\pi^2} \right)\right),\end{align*} and we can conclude
\begin{equation}\label{eq:c4k:asymptotics} \lim_{k \rightarrow \infty} \frac{c_{4k}}{(4k)^3} = \frac{1}{48} \left( \frac 1 3 - \frac{1}{\pi^2} \right).\end{equation}

\subsubsection{$N=4k+2$-gons}\label{subsec:4k+2gons} We now turn to $N=4k+2$-gons. We note that $$a_N = \area(P_N) = \frac{N}{2} \sin \left(\frac{2\pi}{N}\right) = (2k+1) \sin\left(\frac{\pi}{2k+1}\right),$$ and the lengths of the sides $$
r_N = |z_{N,j}| = |e_N(j+1)-e_N(j)|
%r_N = |z_{n,N}| = |e_N(n+1)-e_N(n)| 
= \sqrt{2- 2\cos\left(\frac{2\pi}{N}\right)} = \sqrt{2- 2\cos\left(\frac{\pi}{2k+1}\right)}.$$ As in the $N=4k$ case, we need to compute the constants $c^{(a_N)}_0(S_N)$ and $c_{\pi/N}^{(a_N)}(S_N)$, where we recall from \eqref{eq:veech:even} that \begin{equation}\label{eq:veech:even:notilde}
    c_{\ast}^{(a_N)} = \frac{1}{\pi} \tilde c_{\ast}^{(a_N)} = \frac{1}{\pi}\frac{1}{\covol(\Gamma_N)} 2\cot(\pi/N) \left( \sum_{\gamma \in\SC_{\ast}^{\dagger}(S_N)} \frac{1}{\left|z^{(a_N)}_{\gamma}\right|^2} \right)
\end{equation}

\subsubsection{The horizontal cusp, $N=4k+2$}\label{subsec:4k+2:horizontal} On $S_N^{(a_N)}$, we have horizontal saddle connections $\gamma_j$ connecting $e_N(j)$ to $e_N(2k+1-j)$, $j=0, \pm 1, \ldots, \pm k$. We note that in fact $\gamma_k = \gamma_{-k}$, since each of these corresponds to a side of $P_N$, and $\gamma_0$ has length $2$. Thus, $$\SC_{0}^{\dagger}(S_N) = \{\gamma_0, \gamma_k\} \cup \{\gamma_{\pm j}: j=1, \ldots, k-1\},$$ and lengths are given by \begin{align}\label{eq:lengths:4k+2} \left|z^{(a_N)}_{\gamma_j}\right| &= |e_N(j) - e_N(2k+1-j)| \\ \nonumber &= |e_N(j) - e_N(-j)e(1/2)| \\ \nonumber &= |e_N(j) + e_N(-j)| \mbox{ using } (e(1/2) = -1) \\ \nonumber &= 2\cos\left(2\pi \frac{j}{N} \right) \\ \nonumber &= 2 \cos \left(\frac{j}{2k+1} \pi \right).\end{align} 

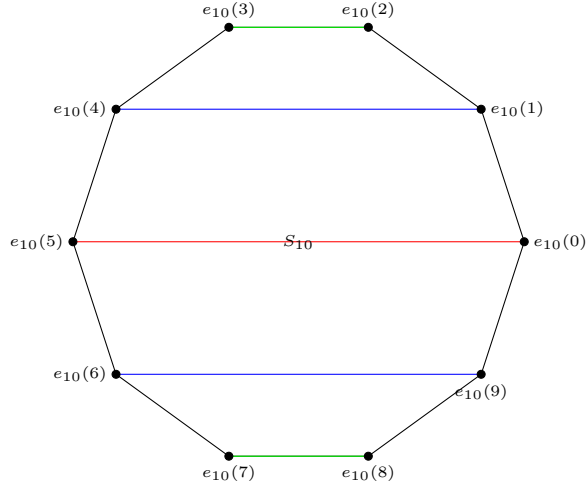
\begin{figure}[ht!]
\begin{tikzpicture}[scale=3.0]
\draw (0:1)node[right]{\tiny $e_{10}(0)$}--(36:1)node[right]{\tiny $e_{10}(1)$}--(72:1)node[above]{\tiny $e_{10}(2)$}--(108:1)node[above]{\tiny $e_{10}(3)$}--(144:1)node[left]{\tiny $e_{10}(4)$}--(180:1)node[left]{\tiny $e_{10}(5)$}--(216:1)node[left]{\tiny $e_{10}(6)$}--(252:1)node[below]{\tiny $e_{10}(7)$}--(288:1)node[below]{\tiny $e_{10}(8)$}--(324:1)node[below]{\tiny $e_{10}(9)$}--cycle;
\draw[red](0:1)--(180:1);
\draw[blue](36:1)--(144:1);
\draw[green](72:1)--(108:1);
\node(0,0.2){\tiny $S_{10}$};
\draw[blue](-36:1)--(-144:1);
\draw[green](-72:1)--(-108:1);
\filldraw [black] (0:1) circle (.5pt);
\filldraw [black] (36:1) circle (.5pt);
\filldraw [black] (72:1) circle (.5pt);
\filldraw [black] (108:1) circle (.5pt);
\filldraw [black] (144:1) circle (.5pt);
\filldraw [black] (180:1) circle (.5pt);
\filldraw [black] (216:1) circle (.5pt);
\filldraw [black] (252:1) circle (.5pt);
\filldraw [black] (288:1) circle (.5pt);
\filldraw [black] (324:1) circle (.5pt);

\end{tikzpicture}
\caption{Horizontal saddle connections on $S_{10} \in \hh(1, 1)$. The \textcolor{red}{red} saddle connection is of length $2$ and the \textcolor{blue}{blue} saddle connections occur in pairs of the same length, and the \textcolor{green}{green} saddle connection is a side of $P_{10}$.}\label{fig:4k+2:horizontal}
\end{figure}

\newpage

Thus \begin{align}\label{eq:lengthsum:4k+2} \sum_{\gamma \in\SC_{0}^{\dagger}(S_N)} \frac{1}{\left|z^{(a_N)}_{\gamma}\right|^2}  &= \frac{1}{\left|z^{(a_N)}_{\gamma_0}\right|^2} + \frac{1}{\left|z^{(a_N)}_{\gamma_k}\right|^2} + 2 \sum_{j=1}^{k-1} \frac{1}{\left|z^{(a_N)}_{\gamma_j}\right|^2} \\ \nonumber &= \left( \frac{1}{2^2} + \frac{1}{4 \cos^2\left(\frac{k}{2k+1}\pi\right)} + 2\sum_{j=1}^{k-1} \frac{1}{4\cos^2\left(\frac{j}{2k+1}\pi\right)}\right) \\ \nonumber &= \left( \frac{1}{4}- \frac{1}{4\cos^2\left(\frac{k}{2k+1}\pi\right)} + \frac{1}{2}\sum_{j=1}^{k} \frac{1}{\cos^2\left(\frac{j}{2k+1}\pi\right)} \right).\end{align} 
So we need to compute the sum \begin{equation}\label{eq:sigma0stark:def} \Sigma_0^{\ast}(k) := \sum_{j=1}^{k} \frac{1}{\cos^2\left(\frac{j}{2k+1}\pi \right)}.\end{equation} We will show that \begin{equation}\label{eq:sigma0stark} \Sigma_0^{\ast}(k) = 2(k^2+k).\end{equation} Since \begin{align*} \cos^2 \left( \frac{j}{2k+1} \pi \right) &= \sin^2\left( \frac{\pi}{2} - \frac{j}{2k+1}\pi\right) \\ &=  \sin^2\left( \frac{2k+1-2j}{4k+2}\pi\right),\end{align*} putting $t=k+1-j$, and noting that $j=1, \ldots, k$ means $t= k, \ldots, 1$, and $2t-1 = 2k+2-2j-1= 2k+1-2j$, we have $$\Sigma_0^{\ast}(k) = \sum_{t=1}^{k} \frac{1}{\sin^2\left( \frac{2t-1}{4k+2}\pi\right)}.$$ By \eqref{eq:veech:identity}, with $m=N=4k+2$, \begin{align} \frac{1}{3}(N^2-1) &= \sum_{j=1}^{4k+1} \frac{1}{\sin^2\left(\frac{j}{N}\pi\right)} \\ \nonumber &= \sum_{l=1}^{2k+1} \frac{1}{\sin^2\left(\frac{2l-1}{N}\pi\right)} + \sum_{l=1}^{2k} \frac{1}{\sin^2\left(\frac{2l}{N}\pi\right)} \mbox{ (splitting into odd and even terms)} \\ \nonumber &= \sum_{l=1}^{2k+1} \frac{1}{\sin^2\left(\frac{2l-1}{N}\pi\right)} + \sum_{l=1}^{2k} \frac{1}{\sin^2\left(\frac{l}{2k+1}\pi\right)} \\ \nonumber &= \sum_{l=1}^{2k+1} \frac{1}{\sin^2\left(\frac{2l-1}{N}\pi\right)}  + \frac 1 3 \left( (2k+1)^2 -1 \right),
\end{align}
where in the last line we are using \eqref{eq:veech:identity} with $m=2k+1$. Thus \begin{align*}\sum_{l=1}^{2k+1} \frac{1}{\sin^2\left(\frac{2l-1}{N}\pi\right)} &=  \frac{1}{3}(N^2-1) - \frac 1 3 \left( (2k+1)^2 -1 \right) \\ &= \frac{1}{3} ( (4k+2)^2 - (2k+1)^2) \\ &= (2k+1)^2.\end{align*} Since  \begin{align*} \sin^2 \left( \frac{2t-1}{N} \pi \right) &= \sin^2 \left( \frac{4k+2-(2t-1)}{N} \pi \right) \\ &= \sin^2 \left( \frac{2(2k+2-t)-1}{N} \pi \right),
\end{align*} and if $t=1, \ldots, k$, $s=2k+2-t$ ranges from $s=2k+1, \ldots, k+2$, and $$\sin^2 \left( \frac{2(k+1)-1}{N} \pi \right) = \sin^2 \left( \frac{2k+1}{N} \pi \right) = \sin^2(\pi/2) = 1,$$ we have \begin{align*} 2\Sigma_0^{\ast}(k)+1 &= 2\sum_{t=1}^{k} \frac{1}{\sin^2\left( \frac{2t-1}{N}\pi\right)} + 1 \\ &= \sum_{l=1}^{2k+1} \frac{1}{\sin^2\left(\frac{2l-1}{N}\pi\right)} \\ &= (2k+1)^2,
\end{align*} where the term $1$ comes from the index $l=k+1$, the indices $l=1,\ldots, k$ correspond to the indices $t$, and the indices $l=k+2, \ldots, 2k+1$ correspond to the indices $s$. Thus $$\Sigma_0^{\ast}(k) = \frac{1}{2}\left((2k+1)^2-1\right) = 2(k^2+k),$$ as we claimed in \eqref{eq:sigma0stark}. Plugging this into \eqref{eq:veech:even:notilde}, we have \begin{align}\label{eq:4k+2:horizontal:final} c_{0}^{(a_N)} &= \frac{1}{\pi}\frac{1}{\covol(\Gamma_N)} 2\cot(\pi/N) \left( \sum_{\gamma \in\SC_{0%\ast
}^{\dagger}(S_N)} \frac{1}{\left|z^{(a_N)}_{\gamma}\right|^2} \right) \\ \nonumber &= \frac{1}{\pi^2}\frac{N}{N-2}\cot\left(\frac{\pi}{N}\right) \left( \frac{1}{4} - \frac{1}{4 \cos^2\left(\frac{k}{2k+1} \pi \right)} + \frac{1}{2} \Sigma^{\ast}_0(k)\right) \\ \nonumber &= \frac{1}{\pi^2}\frac{N}{N-2}\cot\left(\frac{\pi}{N}\right) \left( \frac{1}{4} - \frac{1}{4 \cos^2\left(\frac{k}{2k+1} \pi \right)} + k^2+k\right).
\end{align}

\subsubsection{The $\pi/N$ cusp, $N=4k+2$}\label{subsec:4k+2:piN} In the direction $\pi/N$, we have $2k$ saddle connections $\gamma_j$, $j=-(k-1), \ldots, k$, connecting $e_N(j)$ to $e_N(2k+2-j)$. The saddle connections $\gamma_j$ and $\gamma_{1-j}$ have the same length, since \begin{align*}\left|z^{(a_N)}_{\gamma_{1-j}} \right| &= |e_N(1-j) - e_N(2k+2- (1-j))| \\ &= |e_N(1-j) - e_N(2k+1 +j)| \\ &= |e_N(1-j) - e_N(2k+1)e_N(j)| \\ &= |e_N(1-j) + e_N(j)| \end{align*} and \begin{align*}\left|z^{(a_N)}_{\gamma_j}\right| &= |e_N(j) - e_N(2k+2-j)| \\&= |e_N(j) - e_N(2k+1 + 1- j)| \\ &= |e_N(j) - e_N(2k+1)e_N(1-j)| \\ &= |e_N(j) + e_N(1-j)|.\end{align*} We can refine this length computation by multiplying by $e_N(-1/2)$, so we have \begin{align*} \left|z^{(a_N)}_{\gamma_j}\right| = \left|z^{(a_N)}_{\gamma_{1-j}}\right| &= |e_N(j) + e_N(1-j)|\\ &= |e_N(j-1/2) + e_N(1/2-j)| \\ &= 2\cos \left( \frac{2j-1}{N} \pi \right).    
\end{align*}

\begin{figure}[ht!]
\begin{tikzpicture}[scale=3.0]
\draw (0:1)node[right]{\tiny $e_{10}(0)$}--(36:1)node[right]{\tiny $e_{10}(1)$}--(72:1)node[above]{\tiny $e_{10}(2)$}--(108:1)node[above]{\tiny $e_{10}(3)$}--(144:1)node[left]{\tiny $e_{10}(4)$}--(180:1)node[left]{\tiny $e_{10}(5)$}--(216:1)node[left]{\tiny $e_{10}(6)$}--(252:1)node[below]{\tiny $e_{10}(7)$}--(288:1)node[below]{\tiny $e_{10}(8)$}--(324:1)node[below]{\tiny $e_{10}(9)$}--cycle;
\draw[blue](0:1)--(216:1);
\draw[blue](36:1)--(180:1);
\draw[blue](72:1)--(144:1);
\node(0,0.2){\tiny $S_{10}$};
\draw[blue](-36:1)--(252:1);
\filldraw [black] (0:1) circle (.5pt);
\filldraw [black] (36:1) circle (.5pt);
\filldraw [black] (72:1) circle (.5pt);
\filldraw [black] (108:1) circle (.5pt);
\filldraw [black] (144:1) circle (.5pt);
\filldraw [black] (180:1) circle (.5pt);
\filldraw [black] (216:1) circle (.5pt);
\filldraw [black] (252:1) circle (.5pt);
\filldraw [black] (288:1) circle (.5pt);
\filldraw [black] (324:1) circle (.5pt);

\end{tikzpicture}
\caption{Saddle connections of angle $\pi/10$ on $S_{10} \in \hh(1, 1)$. Note that the indices of the vertices add to $6 \pmod{10}$.}\label{fig:4k+2:piN}
\end{figure}
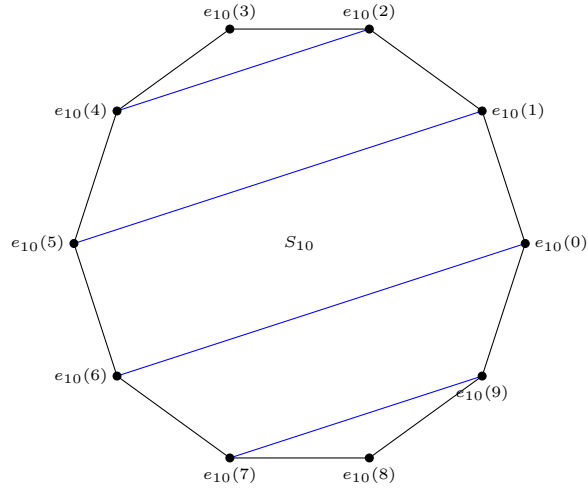
So we have \begin{align}\label{eq:veech:4k+2:piN}
     c_{\pi/N}^{(a_N)} &=  \frac{1}{\pi^2}\frac{N}{N-2}\cot(\pi/N) \left( \sum_{\gamma \in\SC_{\pi/N}^{\dagger}(S_N)} \frac{1}{\left|z^{(a_N)}_{\gamma}\right|^2} \right) \\ \nonumber &= \frac{1}{\pi^2}\frac{N}{N-2}\cot(\pi/N) \left( \sum_{j=1}^k \frac{2}{4 \cos^2 \left( \frac{2j-1}{N} \pi \right)} \right) \\ \nonumber &= \frac{1}{\pi^2}\frac{N}{N-2}\cot(\pi/N) \left( \frac 1 2 \sum_{j=1}^k \frac{1}{ \cos^2 \left( \frac{2j-1}{N} \pi \right)} \right)
\end{align}
Thus, we must compute \begin{equation}\label{eq:sigmapiNk:4k+2:def} \Sigma_{\pi/N}^{\ast}(k) := \sum_{j=1}^k \frac{1}{ \cos^2 \left( \frac{2j-1}{N} \pi \right)}.
\end{equation}
We claim  \begin{equation}\label{eq:sigmapiNk:4k+2:formula} \Sigma_{\pi/N}^{\ast}(k) = \frac 2 3 (k^2 + k). \end{equation}
To prove this, first note for $t=k-j+1$, as $j$ ranges from $1, \ldots, k$, $t$ ranges from $k, \ldots, 1$, and $$\frac{2j-1}{N} = \frac{2(k-t+1)-1}{N} = \frac{2k+1-2t}{N} = \frac{1}{2} - \frac{2t}{N},$$ so \begin{align*}\cos^2 \left( \frac{2j-1}{N} \pi \right) &= \cos^2 \left( \frac{\pi}{2}- \frac{2t}{N} \pi \right) \\ &=  \sin^2 \left( \frac{2t}{N} \pi \right) \\& = \sin^2\left(\frac{t}{2k+1}\pi\right).\end{align*} Since $$\sin^2\left(\frac{2k+1-t}{2k+1}\pi\right)=\sin^2\left(\frac{t}{2k+1}\pi\right),$$ and as $t$ ranges from $1, \ldots, k$, $2k+1-t$ ranges from $2k, \ldots, k+1$, we have $$2\Sigma_{\pi/N}^{\ast}(k) = \sum_{t=1}^{2k} \frac{1}{ \sin^2 \left( \frac{t}{2k+1} \pi\right)} = \frac{1}{3}\left( (2k+1)^2-1\right),$$ where in the last equality we are using \eqref{eq:veech:identity} with $m=2k+1$, which proves \eqref{eq:sigmapiNk:4k+2:formula}. So we have \begin{align}\label{eq:veech:4k+2:piN:final}
     c_{\pi/N}^{(a_N)} &= \frac{1}{\pi^2}\frac{N}{N-2}\cot(\pi/N) \left( \frac 1 2 \sum_{j=1}^k \frac{1}{ \cos^2 \left( \frac{2j-1}{N} \pi \right)} \right) \\ \nonumber &= \frac{1}{\pi^2}\frac{N}{N-2}\cot(\pi/N)\left( \frac{1}{3}(k^2+k)\right).
\end{align}

\subsubsection{Combinatorial constants and $N=4k+2 \rightarrow \infty$ asymptotics}\label{subsec:4k+2:final} Combining \eqref{eq:4k+2:horizontal:final} and \eqref{eq:veech:4k+2:piN:final}, we have for $N=4k+2$, \begin{align}\label{eq:veech:4k+2:final}
    c_{\cC_{N}}(S_{N})  &= a_{N}^2 c^{(a_N)}_{\cG} \left|\hat\Omega^{(a_N)}_{\cC}\right| \\ \nonumber &= \frac{N^2}{4} \sin^2\left(\frac{2\pi}{N}\right) \frac{N}{4} \tan\left(\frac{\pi}{N}\right) \left(c^{(a_N)}_0(S_N) + c_{\pi/N}^{(a_N)}(S_N) \right) \\ \nonumber &= \frac{N^3}{16} \sin^2\left(\frac{2\pi}{N}\right) \tan\left(\frac{\pi}{N}\right) \frac{1}{\pi^2} \cot\left(\frac{\pi}{N}\right) \frac{N}{N-2} \left( \left( \frac{1}{4} - \frac{1}{4 \cos^2\left(\frac{k}{2k+1} \pi \right)} + k^2+k\right) + \frac{1}{3}(k^2+k) \right) \\ \nonumber &= \frac{N^4 \sin^2\left(\frac{2\pi}{N}\right)}{16\pi^2(N-2)}\left( \left( \frac{1}{4} - \frac{1}{4 \cos^2\left(\frac{k}{2k+1} \pi \right)}\right) + \frac{4}{3}(k^2+k) \right) \\ \nonumber &= \frac{N^4 \sin^2\left(\frac{2\pi}{N}\right)}{16\pi^2(N-2)}\left( \left( \frac{1}{4} - \frac{1}{4 \sin^2\left(\frac{\pi}{N}  \right)}\right) + \frac{4}{3}(k^2+k) \right),\end{align} where in the last line we are using $$\cos^2\left(\frac{k}{2k+1} \pi \right) = \sin^2\left(\frac{\pi}{2}- \frac{k\pi}{2k+1}  \right) = \sin^2\left(\frac{\pi}{4k+2}  \right)=\sin^2\left(\frac{\pi}{N}  \right).$$
Since $N^2 = 16k^2 + 16k + 4,$ $$k^2+k = \frac{N^2}{16}- \frac{1}{4},$$ so applying \eqref{eq:complexity:diagonals}, we have, for $N=4k+2$, \begin{align}\label{eq:4k+2:complexity} c_N = \frac{1}{3} c_{\cC_{N}}(S_{N}) &=  \frac{N^4 \sin^2\left(\frac{2\pi}{N}\right)}{48\pi^2(N-2)}\left( \left( \frac{1}{4} - \frac{1}{4 \sin^2\left(\frac{\pi}{N}  \right)}\right) + \frac{4}{3}\left(\frac{N^2}{16}- \frac{1}{4}\right) \right) \\ \nonumber &=\frac{N^4 \sin^2\left(\frac{2\pi}{N}\right)}{48\pi^2(N-2)}\left( \frac{1}{12}N^2 - \frac{1}{4 \sin^2\left(\frac{\pi}{N}  \right)}- \frac{1}{12}\right) .\end{align}
As in the $N=4k$ case, we can conclude
\begin{equation}\label{eq:c4k+2:asymptotics} \lim_{k \rightarrow \infty} \frac{c_{4k+2}}{(4k+2)^3} = \frac{1}{48} \left( \frac 1 3 - \frac{1}{\pi^2} \right).\end{equation}

\subsection{Odd-gons}\label{subsec:odd} We now turn to odd $N=2k+1$, and to computing the asymptotic complexity of the billiard $B_N$. As discussed in \S\ref{subsec:covers}, the associated lattice surface $S_N^{(2a_N)}$ we will be working with is obtained by identifying opposite sides of a \emph{doubled} regular $N$-gon. For concreteness, we will consider one copy of $P_N$ inscribed in the unit circle with vertices at $e_N(j), j=0, \ldots, N-1=2k$, and the other copy obtained by reflecting the initial copy over the vertical side, connecting $e_N(k)$ to $e_N(k+1)$ (see Figure~\ref{fig:2k+1:vertical} below). The Veech group $\Gamma(S_N) = \Gamma_N$ is given by $\Gamma_N = \Delta^+(2, N, \infty)$, and  has  covolume $$\covol(\Gamma_N) = 2\left(\pi- \frac{\pi}{2} - \frac{\pi}{N}\right) = \pi\frac{N-2}{N}. $$ We have \begin{align}\label{eq:odd:combconst}c_{\cC_N}(S_N) &= \lim_{t \rightarrow \infty} \frac{N_{\cC_N}(S_N, t)}{t^2} \\ &= (2a_N) \left|\hat{\Omega}_{\cC_N}^{(2a_N)}\right| \cone_{\cG}(S_N)\\ \nonumber  &= (2a_N) \left|\hat{\Omega}_{\cC_N}^{(2a_N)}\right| (2a_N) c^{(2a_N)}_{\cG}(S_N).\end{align}
In the $N=2k+1$-setting, the Veech group $\Gamma_N$ only has one cusp, and thus we can take any saddle connection direction as its representative. It will be convenient to use the \emph{vertical} direction, that is, we have $\SC(S_N) = \Gamma_N \cdot \SC_{\pi/2}^{\dagger}(S_N),$ where $\SC_{\pi/2}^{\dagger}(S_N)$ is the set of saddle connections on $S_N^{(2a_N)}$ of angle $\pi/2$ with the horizontal. Following Veech~\cite{Veech:eisenstein}, we have \begin{equation}\label{eq:veech:counting:odd} c_{\cG}^{(2a_N)} = \frac{1}{\pi}\frac{1}{\covol(\Gamma_N)} 2\cot(\pi/N) \left( \sum_{\gamma \in \SC_{\pi/2}^{\dagger}\left(S_N^{(2a_N)}\right)} \frac{1}{\left|z^{(2a_N)}_{\gamma}\right|^2} \right). \end{equation}

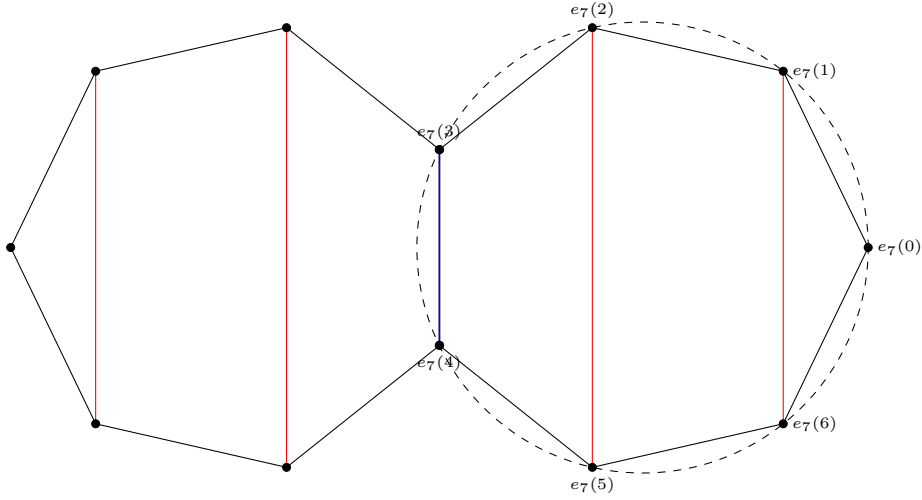
\begin{figure}
\begin{tikzpicture}[scale=3.0]
\draw (0:1)node[right]{\tiny $e_7(0)$} --(360/7:1)node[right]{\tiny $e_7(1)$}--(720/7:1)node[above]{\tiny $e_7(2)$}--(1080/7:1)node[above]{\tiny $e_7(3)$}--(1440/7:1)node[below]{\tiny $e_7(4)$}--(1800/7:1)node[below]{\tiny $e_7(5)$}--(2160/7:1)node[right]{\tiny $e_7(6)$}--cycle;
\draw[red](360/7:1)--(2160/7:1);
\draw[red](720/7:1)--(1800/7:1);
\draw[blue](1080/7:1)--(1440/7:1);
    \filldraw [black] (0:1) circle (.5pt);
\filldraw [black] (360/7:1) circle (.5pt);
\filldraw [black] (720/7:1) circle (.5pt);
\filldraw [black] (1080/7:1) circle (.5pt);
\filldraw [black] (1440/7:1) circle (.5pt);
\filldraw [black] (1800/7:1) circle (.5pt);
\filldraw [black] (2160/7:1) circle (.5pt);
\draw[dashed](0,0) circle(1);

\begin{scope}[shift={(-1.8,0)},rotate=180]
    \draw (0:1)--(360/7:1)--(720/7:1)--(1080/7:1)--(1440/7:1)--(1800/7:1)--(2160/7:1)--cycle;
    \draw[red](360/7:1)--(2160/7:1);
\draw[red](720/7:1)--(1800/7:1);
\draw[blue](1080/7:1)--(1440/7:1);
    \filldraw [black] (0:1) circle (.5pt);
\filldraw [black] (360/7:1) circle (.5pt);
\filldraw [black] (720/7:1) circle (.5pt);
\filldraw [black] (1080/7:1) circle (.5pt);
\filldraw [black] (1440/7:1) circle (.5pt);
\filldraw [black] (1800/7:1) circle (.5pt);
\filldraw [black] (2160/7:1) circle (.5pt);
\end{scope}
\end{tikzpicture}
\caption{Vertical saddle connections on the surface $S_{7} \in \hh(4)$. Note that each of the \textcolor{red}{red} saddle connections occurs in pairs of the same length, while the \textcolor{blue}{blue} saddle connection is shared between the two copies of $P_7$ which make up $S_7$.}\label{fig:2k+1:vertical}
\end{figure}
\paragraph*{\bf Vertical saddle connections} The vertical saddle connections are given by $\gamma^{\pm}_j, j=1, \ldots, k$ on $S^{(2a_N)}_N$, which connect (in each copy of $P_N$) $e_N(j)$ to $e_N(N-j)$, with the $j=k$ copy shared between the two copies (that is $\gamma^+_k = \gamma^-_k$). We have $$\left|z^{(2a_N)}_{\gamma_j}\right| = |e_N(j) - e_N(N-j)|= 2\sin \left(2\pi \frac{j}{N} \right).$$ So we have \begin{align}\label{eq:veech:cusp:odd} c_{\cG}^{(2a_N)} &= \frac{1}{\pi}\frac{1}{\covol(\Gamma_N)} 2\cot(\pi/N) \left( \sum_{\gamma \in \SC_{\pi/2}^{\dagger}\left(S_N^{(2a_N)}\right)} \frac{1}{\left|z^{(2a_N)}_{\gamma}\right|^2} \right) \\ \nonumber &= \frac{1}{\pi^2} \frac{N}{N-2} 2\cot\left(\frac{\pi}{N}\right) \left( \sum_{j=1}^{k-1} \frac{2}{4 \sin^2\left(2\pi \frac{j}{N} \right)} + \frac{1}{4\sin^2\left(2\pi \frac{k}{N} \right)}\right) \\ \nonumber& = \frac{1}{\pi^2} \frac{N}{N-2} 2\cot\left(\frac{\pi}{N}\right) \left( \frac 1 2 \sum_{j=1}^{k} \frac{1}{ \sin^2\left(2\pi \frac{j}{N} \right)} - \frac{1}{4\sin^2\left(2\pi \frac{k}{N} \right)}\right) \nonumber \\ &= \frac{1}{\pi^2} \frac{N}{N-2} \cot\left(\frac{\pi}{N}\right) \left( \sum_{j=1}^{k} \frac{1}{ \sin^2\left(2\pi \frac{j}{N} \right)} - \frac{1}{2\sin^2\left(2\pi \frac{k}{N} \right)}\right). \end{align} So we need to compute the sum \begin{equation}\label{eq:sigmapi2starstar:def} \Sigma^{\ast\ast}_{\pi/2}(k):= \sum_{j=1}^{k} \frac{1}{ \sin^2\left(2\pi \frac{j}{2k+1} \right)}.
\end{equation}
We claim \begin{equation}\label{eq:sigmapi2starstar:value} \Sigma^{\ast\ast}_{\pi/2}(k) = \frac{2}{3}(k^2+k).
\end{equation}
To see this, note that for $t=k-j+1$, we have, as $j$ ranges from $1, \ldots, k$, $t$ ranges from $k, \ldots, 1$, and $$\sin^2\left(\pi \frac{2j}{N} \right) = \sin^2\left(\pi \frac{2t-1}{N} \right),$$ since \begin{align*}\frac{\pi}{N} (2t-1) &= \frac{\pi}{N} \left( 2(k-j+1)-1 \right) \\&= \pi - \pi \frac{2j}{N}. \end{align*} Thus \begin{align*}
 2\Sigma_{\pi/2}^{\ast\ast} &=    \sum_{j=1}^{k} \frac{1}{ \sin^2\left(\pi \frac{2j}{2k+1} \right)} + \sum_{t=1}^{k} \frac{1}{ \sin^2\left(\pi \frac{2t-1}{2k+1} \right)} \\ &= \sum_{j=1}^{2k} \frac{1}{ \sin^2\left(\pi \frac{j}{2k+1} \right)} \\ &= \frac{1}{3} \left( (2k+1)^2 - 1 \right),
\end{align*}
where in the last line we are using \eqref{eq:veech:identity} with $m=2k+1$. Thus, \begin{equation}\label{eq:comb:oddgon} c_{\cG}^{(2a_N)} = 
    \frac{1}{\pi^2} \frac{N}{N-2} \cot\left(\frac{\pi}{N}\right) \left( \frac 2 3 (k^2+k) - \frac{1}{2\sin^2\left(2\pi \frac{k}{N} \right)}\right).
\end{equation}
It remains to compute $\left|\hat{\Omega}_{\cC_N}^{(2a_N)}\right|,$ where $$\hat{\Omega}_{\cC_N}^{(2a_N)} = \left\{z \in \C: \sum_{j=0}^{N-1} |z \wedge z_{N,j}| \le 1 \right\},$$ where $z_{N, j} = e_N(j+1) - e_N(j).$ Direct computation shows that $\hat{\Omega}_{\cC_N}^{(2a_N)}$ is a regular $2N$-gon inscribed in a circle of radius $$\rho_N = \left(\sum_{j=0}^{N-1}\left|\cos \left( 2\pi \frac{j+1}{N} \right) - \cos \left( 2\pi \frac{j}{N} \right)\right| \right)^{-1}.$$ Using the fact that $\cos$ is decreasing on $(0, \pi)$, symmetry, and telescoping sums, we have \begin{align*} \rho_N^{-1} &= \sum_{j=0}^{N-1}\left|\cos \left( 2\pi \frac{j+1}{N} \right) - \cos \left( 2\pi \frac{j}{N}\right)\right| \\ &= -2 \sum_{j=0}^{k-1} \left(\cos \left( 2\pi \frac{j+1}{N} \right) - \cos \left( 2\pi \frac{j}{N}\right)\right) \\& = 2\left(1- \cos\left(2\pi \frac{k}{2k+1}\right)\right).
\end{align*}
Thus \begin{equation}\label{eq:omegahat2an:odd} \left|\hat{\Omega}_{\cC_N}^{(2a_N)}\right| = \rho_N^2 a_{2N} = \frac{1}{4}\left(1- \cos\left(2\pi \frac{k}{2k+1}\right)\right)^{-2} \frac{2N}{2}\sin\left( \frac{2\pi}{2N}\right),
\end{equation} where $a_{2N}$ is the area of a regular $2N$-gon inscribed in the unit circle.

\subsubsection{Combinatorial constants and $N=2k+1 \rightarrow \infty$ asymptotics}\label{subsec:2k+1:final} Putting this all together, we have

\begin{align}\label{eq:veech:2k+1:final} c_{\cC_N}(S_N) &= (2a_N) \left|\hat{\Omega}_{\cC_N}^{(2a_N)}\right| (2a_N) c^{(2a_N)}_{\cG}(S_N) \\ \nonumber &= 4a_N^2 \left|\hat{\Omega}_{\cC_N}^{(2a_N)}\right| c^{(2a_N)}_{\cG}(S_N) \\ \nonumber &= 4 \frac{N^2}{4} \sin^2\left(\frac{2\pi}{N}\right)   \frac{1}{\pi^2} \frac{N}{N-2} \cot\left(\frac{\pi}{N}\right) \left( \frac 2 3 (k^2+k) - \frac{1}{2\sin^2\left(2\pi \frac{k}{N} \right)}\right) \\ \nonumber & \cdot \frac{1}{4}\left(1- \cos\left(2\pi \frac{k}{2k+1}\right)\right)^{-2} \frac{2N}{2}\sin\left( \frac{2\pi}{2N}\right) \\ \nonumber &= \frac{1}{4\pi^2}\frac{N^4}{N-2} \sin^2\left(\frac{2\pi}{N}\right)  \cot\left(\frac{\pi}{N}\right)\sin\left( \frac{\pi}{N}\right)\frac{\left( \frac 2 3 (k^2+k) - \frac{1}{2\sin^2\left(2\pi \frac{k}{N} \right)}\right)}{\left(1- \cos\left(2\pi \frac{k}{2k+1}\right)\right)^{2}},
\end{align}
and
\begin{align}\label{eq:complexity:2k+1:final} c_{N} &= \frac{1}{6} c_{\cC_N}(S_N) \\ \nonumber  &=  \frac{1}{24\pi^2} \frac{N^4}{N-2} \sin^2\left(\frac{2\pi}{N}\right)  \cot\left(\frac{\pi}{N}\right)\sin\left( \frac{\pi}{N}\right)\frac{\left( \frac 2 3 (k^2+k) - \frac{1}{2\sin^2\left(2\pi \frac{k}{N} \right)}\right)}{\left(1- \cos\left(2\pi \frac{k}{2k+1}\right)\right)^{2}}.
\end{align}
Since \begin{align*} \sin^2\left(2\pi \frac{k}{N} \right) &= \sin^2\left(\frac{\pi}{N}\right) \\ \cos\left(2\pi \frac{k}{2k+1}\right) &=- \cos \left( \frac{\pi}{N} \right)\\ \cot \left( \frac{\pi}{N} \right) \sin \left( \frac{\pi}{N} \right) & = \cos \left( \frac{\pi}{N} \right) \\ k^2+k &= \frac{1}{4} \left(N^2-1\right), \end{align*} we can simplify the above expression into 
\begin{equation}\label{eq:complexity:2k+1:final:simplified} c_{N} = \frac{N^4 \sin^2\left(\frac{2\pi}{N}\right)}{48\pi^2(N-2)} \left( \frac{1}{12} N^2 - \frac{1}{4\sin^2\left(\frac{\pi}{N}\right)} - \frac{1}{12}\right) \frac{4\cos \left( \frac{\pi}{N} \right)}{\left(1+\cos \left( \frac{\pi}{N} \right)\right)^2}\end{equation}
As in the setting of $N$ even, we have $$\lim_{N\rightarrow \infty} \frac{\frac{N^4 \sin^2\left(\frac{2\pi}{N}\right)}{48\pi^2(N-2)} \left( \frac{1}{12} N^2 - \frac{1}{4\sin^2\left(\frac{\pi}{N}\right)} - \frac{1}{12}\right)}{N^3}= \frac{1}{48} \left( \frac 1 3 - \frac{1}{\pi^2} \right),$$ and since $$\lim_{N \rightarrow \infty} \frac{4\cos \left( \frac{\pi}{N} \right)}{\left(1+\cos \left( \frac{\pi}{N} \right)\right)^2} = 1$$ we have $$\lim_{k \rightarrow \infty} \frac{c_{2k+1}}{(2k+1)^3} = \frac{1}{48} \left( \frac 1 3 - \frac{1}{\pi^2} \right),$$ completing the proof of \eqref{eq:cN:asymptotics}, and thus of Theorem~\ref{theorem:billiardcomplexity:Ngons}.
\appendix
%\section{Constants for rational lattice polygons}\label{appendix:constants}

\end{document}